\documentclass{amsart}[11pt]
 \usepackage{amsmath}
\usepackage{amsfonts}
\usepackage{amssymb}
\usepackage{a4wide}

\usepackage{graphicx}
\usepackage{fullpage}
\usepackage[colorlinks=true]{hyperref}
\usepackage{color}
\usepackage{epsfig}

\newtheorem{lemma}{Lemma}

\newtheorem{theorem}{Theorem}
\newtheorem{example}{Example}[section]
\newtheorem{remark}{Remark}[section]

\begin{document}

\title{\bf  Multicut decomposition methods with cut selection for multistage stochastic programs}

\maketitle

\vspace*{0.5cm}

\begin{center}
\begin{tabular}{cc}
\begin{tabular}{c}
Vincent Guigues (corresponding author)\\
School of Applied Mathematics, FGV\\
Praia de Botafogo, Rio de Janeiro, Brazil\\ 
{\tt vguigues@fgv.br}\\
\end{tabular}
\begin{tabular}{c}
Michelle Bandarra\\
School of Applied Mathematics, FGV\\
Praia de Botafogo, Rio de Janeiro, Brazil\\ 
{\tt michelle.bandarra@mirow.com.br}\\
\end{tabular}
\end{tabular}
\end{center}

\vspace*{0.5cm}





\date{}

\begin{abstract}
We introduce a variant of Multicut Decomposition Algorithms (MuDA), called CuSMuDA (Cut Selection for Multicut Decomposition Algorithms), for solving multistage stochastic linear programs that incorporates
strategies to select the most relevant cuts of the approximate recourse functions. 
We prove the convergence of the method in a finite number
of iterations and use it to solve six portfolio problems with direct transaction costs under return uncertainty and
six inventory management problems under demand uncertainty. On all problem instances CuSMuDA is much quicker than MuDA: between
5.1 and 12.6 times quicker for the porfolio problems considered and between 6.4 and 15.7 times quicker for the inventory problems.\\
\end{abstract}

\par {\textbf{Keywords:}} Stochastic programming; Stochastic Dual Dynamic Programming; Multicut Decomposition Algorithm; Portfolio selection; Inventory management.\\

\par {\textbf{AMS subject classifications:}} 90C15, 91B30.

\section{Introduction}

Multistage stochastic optimization problems are common in many areas of engineering and in finance.
However, solving these problems is challenging and in general requires decomposition techniques.
Two popular decomposition methods are Approximate Dynamic Programming (ADP) (Powell 2011) 
and sampling-based variants of the Nested Decomposition (ND) algorithm (Birge 1985, Birge and Louveaux 1997)
and of the Multicut Nested Decomposition (MuND) algorithm (Birge and Louveaux 1988, Gassmann 1990).
The introduction of sampling within ND was proposed by (Pereira and Pinto 1991) and the corresponding
method is usually called Stochastic Dual Dynamic Programming (SDDP).
Several enhancements and extensions of SDDP have been proposed such as 
CUPPS (Chen and Powell 1999), ReSA (Hindsberger and Philpott 2001), AND (Birge and Donohue 2001), DOASA (Philpott and Guan 2008),
risk-averse variants in (Guigues and R\"omisch 2012, Guigues and R\"omisch 2012, Philpott and de Matos 2012, 
Guigues 2013, Shapiro et al. 2013, Shapiro et al. 2013, Kozmik and Morton 2015),
cut formulas for nonlinear problems in (Guigues 2016), regularizations
in (Asamov and Powell 2015) for linear problems
and SDDP-REG in (Guigues et al. 2017) for nonlinear problems.  
Convergence of these methods was proved in (Philpott and Guan 2008) for linear problems,
in (Girardeau et al. 2015) for nonlinear risk-neutral problems, and in (Guigues 2016) for risk-averse nonlinear problems. 
All these algorithms compute lower approximations of the cost-to-go functions expressed as a supremum of affine functions called
optimality cuts. Typically, at each iteration, a fixed (moderate to large) number of cuts is added for each cost-to-go function.
Therefore cut selection or cut pruning strategies may be helpful to speed up the convergence.
In stochastic optimization, the problem of cut selection for lower approximations of the cost-to-go functions
associated to each node of the scenario tree was discussed for the first time in (Ruszczy\'nski 1993) where only
the active cuts are selected.
Pruning strategies of basis (quadratic) functions have been proposed in (Gaubert et al. 2011) and 
(McEneaney et al. 2008) for max-plus based approximation methods which, similarly to SDDP, approximate the cost-to-go functions of a nonlinear
optimal control problem by a supremum of basis functions. More precisely, in 
(Gaubert et al. 2011), a fixed number of cuts is eliminated and cut selection is done solving a combinatorial optimization problem.
For SDDP, in (Shapiro et al. 2013a) it is suggested at some iterations to eliminate redundant cuts
(a cut is redundant if it is never active in describing the lower approximate cost-to-go function).
This procedure is called {\em{test of usefulness}} in (Pfeiffer et al. 2012).
This requires solving at each stage as many linear programs as there are cuts. 
In (Pfeiffer et al. 2012) and (Philpott et al. 2012), only the cuts that have the largest value for at least
one of the trial points computed are considered relevant. 
This 
strategy is called the {\em{Territory algorithm}} in (Pfeiffer et al. 2012)
and Level 1 cut selection in (Philpott et al. 2012). It was presented for the first time
in 2007 at the ROADEF congress by David Game and Guillaume Le Roy (GDF-Suez), see (Pfeiffer et al. 2012).

Sampling can also be incorporated into  MuND algorithm.
It was used for instance  recently in (Zhang et al. 2016) to solve a water allocation problem under uncertainty.
This algorithm builds many more cuts  than SDDP per iteration (more than 10 times more in typical implementations) and therefore each iteration takes
more time but much less iterations are needed to satisfy some stopping criterion.
Therefore cut selection strategies could also be useful. However, to the best of our knowledge,
the combination of Multicut decomposition methods 
with cut selection strategies, refereed to as CuSMuDA (Cut Selection for Multicut Decomposition Algorithms) in the sequel, has not been proposed so far. 
In this context, the objective of this paper is to study CuSMuDA and to propose cut selection strategies more efficient than the 
aforementioned ones. More precisely, instead of selecting all the cuts that are the highest at the trial points, we introduce a set of selectors
that select some subset of these cuts. The selectors have to satisfy an assumption (Assumption (H3), see Section \ref{sec:cutselectionalgo})
to ensure the convergence of CuSMuDA in a finite number of iterations. 
We obtain a family of cut selection strategies; a given strategy corresponding to a choice of selectors along the iterations.
In this family, the most economic (in terms of memory) cut selection strategy satisfying (H3),
called Multicut Limited Memory Level 1, selects at each trial point only one cut, namely the oldest cut. 
This strategy is the Multicut variant of the Limited Memory Level 1 cut selection proposed in (Guigues 2017).
The least economic strategy, i.e., the one that keeps the largest amount of cuts, is the Multicut
variant of Level 1. "Between" these two strategies, using the flexibility offered by the selectors (as long as Assumption (H3) is satisfied by these selectors), we obtain
a (large) family of cut selection strategies.
We also prove the convergence of CuSMuDA, that uses cut selection strategies from this family, in a finite number of iterations.
 This proof extends the theory in (Guigues 2017) in two aspects:
(i) first the stochastic case is considered, i.e., SDDP and multicut SDDP are considered whereas the deterministic case, i.e., DDP, was considered in (Guigues 2017) and
(ii) second more general cut selection strategies are dealt with. Item (ii) requires an additional technical discussion to show that
we do not cycle, see Lemma \ref{lemmafinitecuts}.

Numerical experiments on six instances of a portfolio problem
and six instances of an inventory problem
show that the corresponding variant of CuSMuDA is much quicker than sampling-based MuND combined with Multicut Level 1 (between 10.2  and 25.6 times quicker) and much quicker than MuND with sampling in the forward pass
(between 5.1 and 15.7 times quicker).

The outline of the study is as follows. The class of problems considered and assumptions are discussed in Section \ref{classassumptions}.
In Subection \ref{msda} we recall sampling-based MuND while in Subsection \ref{sec:cutselectionalgo} CuSMuDA is described.
Section \ref{sec:convanalysis} proves the convergence of CuSMuDA in a finite number of iterations.
Finally, numerical experiments are presented in Section \ref{sec:numsim}.

%
%
%
%
%
%
%
%
%

\section{Problem formulation and assumptions}\label{classassumptions}

We are interested in solution methods for linear stochastic dynamic programming equations:
the first stage problem is 
\begin{equation}\label{firststodp}
\mathcal{Q}_1( x_0 ) = \left\{
\begin{array}{l}
\inf_{x_1} c_1^T x_1 + \mathcal{Q}_2 ( x_1 )\\
A_{1} x_{1} + B_{1} x_{0} = b_1,
x_1 \geq 0
\end{array}
\right.
\end{equation}
for $x_0$ given and for $t=2,\ldots,T$, $\mathcal{Q}_t( x_{t-1} )= \mathbb{E}_{\xi_t}[ \mathfrak{Q}_t ( x_{t-1}, \xi_{t}  )  ]$ with
\begin{equation}\label{secondstodp}
\mathfrak{Q}_t ( x_{t-1}, \xi_{t}  ) = 
\left\{ 
\begin{array}{l}
\inf_{x_t \in \mathbb{R}^n} c_t^T x_t + \mathcal{Q}_{t+1} ( x_t )\\
A_{t} x_{t} + B_{t} x_{t-1} = b_t,
x_t \geq 0,
\end{array}
\right.
\end{equation}
with the convention that $\mathcal{Q}_{T+1}$ is null and
where for $t=2,\ldots,T$, random vector $\xi_t$ corresponds to the concatenation of the elements in random matrices $A_t, B_t$ which have a known
finite number of rows and random vectors $b_t, c_t$
(it is assumed that $\xi_1$ is not random). For convenience, we will denote 
$$
X_t(x_{t-1}, \xi_t):=\{x_t \in \mathbb{R}^n : A_{t} x_{t} + B_{t} x_{t-1} = b_t, \,x_t \geq 0 \}.
$$
We make the following assumptions:
\begin{itemize}
\item[(H1)] The random vectors $\xi_2, \ldots, \xi_T$ are independent and have discrete distributions with finite support. 
\item[(H2)] The set $X_1(x_{0}, \xi_1 )$ is nonempty and bounded and for every $x_1 \in X_1(x_{0}, \xi_1 )$,
for every $t=2,\ldots,T$, for every realization $\tilde \xi_2, \ldots, \tilde\xi_t$ of $\xi_2,\ldots,\xi_t$,
for every $x_{\tau} \in X_{\tau}( x_{\tau-1} , \tilde \xi_{\tau}), \tau=2,\ldots,t-1$, the set $X_t( x_{t-1} , {\tilde \xi}_t )$
is nonempty and bounded.
\end{itemize}
We will denote by $\Theta_t = \{\xi_{t 1},\ldots,\xi_{t M_t} \}$ the support  of $\xi_t$ for stage $t$ with
$p_{t i}= \mathbb{P}(\xi_t = \xi_{t i}) >0, i=1,\ldots,M_t$ and with vector $\xi_{t j}$ being the concatenation
of the elements in $A_{t j}, B_{t j}, b_{t j}, c_{t j}$.

\section{Algorithms}

\subsection{Multicut stochastic decomposition}\label{msda}

The multicut decomposition method approximates function  $\mathfrak{Q}_t(\cdot, \xi_{t j})$ at iteration $k$ for $t=2,\ldots,T$, $j=1,\ldots,M_t$, 
 by a piecewise affine lower bounding function $\mathfrak{Q}_t^k(\cdot, \xi_{t j})$ 
which is a maximum of affine functions $\mathcal{C}_{t j}^i$ called cuts:
$$
\mathfrak{Q}_t^k(  x_{t-1}, \xi_{t j}) = \max_{1 \leq i \leq k} \mathcal{C}_{t j}^i( x_{t-1}  ) \mbox{ with }\mathcal{C}_{t j}^i (x_{t-1})=\theta_{t j}^i + \langle \beta_{t j}^i , x_{t-1} \rangle
$$
where coefficients $\theta_{t j}^i, \beta_{t j}^i$ are computed as explained below
and where the usual scalar product in $\mathbb{R}^n$ is denoted by $\langle x, y\rangle = x^T y$ for $x, y \in \mathbb{R}^n$. These approximations
provide the lower bounding functions
\begin{equation}\label{approxQt}
\mathcal{Q}_{t}^k ( x_{t-1} )=  \sum_{j=1}^{M_t} p_{t j} \mathfrak{Q}_{t}^k (x_{t-1} , \xi_{t j})
\end{equation}
of $\mathcal{Q}_{t}$. Since $\mathcal{Q}_{T+1}$ is the null function, we will also define
$\mathcal{Q}_{T+1}^k \equiv 0$. The steps in MuDA are described below.\\

\par {\textbf{Step 1: Initialization.}} For $t=2,\ldots,T$, $j=1,\ldots,M_t$, take $\mathfrak{Q}_t^0(\cdot, \xi_{t j})\equiv -\infty$ as initial approximations 
(if known lower bounding affine functions for  $\mathfrak{Q}_t(\cdot, \xi_{t j})$ are available, they can be taken for $\mathfrak{Q}_t^0(\cdot, \xi_{t j})$
modifying the algorithm correspondingly).
Set the iteration count $k$ to 1 and $\mathcal{Q}_{T+1}^0 \equiv 0$.\\
\par {\textbf{Step 2: Forward pass.}} We generate a sample
${\tilde \xi}^k = (\tilde \xi_1^k, \tilde \xi_2^k,\ldots,\tilde \xi_T^k)$ from the distribution of $(\xi_1,\xi_2,\ldots,\xi_T)$,
with the convention that $\tilde \xi_1^k = \xi_1$ (here and in what follows, the tilde symbol will be used to represent realizations of random variables:
for random variable $\xi$, $\tilde \xi$ is a realization of $\xi$). Using approximation $\mathfrak{Q}_t^{k-1}(\cdot, \xi_{t j})$
of $\mathfrak{Q}_t(\cdot, \xi_{t j})$  (computed at previous iterations), we solve the problem
\begin{equation}\label{pbforwardpass}
\left\{
\begin{array}{l}
\inf_{x_t \in \mathbb{R}^n} x_t^T {\tilde c}_t^k + \mathcal{Q}_{t+1}^{k-1} ( x_t )\\
x_t \in X_t(x_{t-1}^k, {\tilde \xi}_{t}^k )
\end{array}
\right.
\end{equation}
for $t=1,\ldots,T$,
where $x_0^k=x_0$ and  $\mathcal{Q}_{t+1}^{k-1}$ is given by \eqref{approxQt} with $k$ replaced by $k-1$. Let $x_t^k$ be an optimal solution of the problem.\\
\par {\textbf{Step 3: Backward pass.}} 
The backward pass builds cuts for $\mathfrak{Q}_t(\cdot, \xi_{t j})$ at $x_{t-1}^k$ computed in the forward pass.
For $k \geq 1$ and $t=1,\ldots,T$, we introduce
the function ${\underline{\mathfrak{Q}}}_t^k : \mathbb{R}^n {\small{\times}} \Theta_t \rightarrow \mathbb{R}$ given by
\begin{equation}\label{backwardt0}
{\underline{\mathfrak{Q}}}_t^k (x_{t-1} , \xi_t   ) =  
\left\{
\begin{array}{l}
\inf_{x_t \in \mathbb{R}^n} c_t^T x_t + \mathcal{Q}_{t+1}^k ( x_t )\\
x_t \in X_t(x_{t-1}, \xi_{t} ),
\end{array}
\right.
\end{equation}
with the convention that $\Theta_1 = \{\xi_1\}$, and
we set $\mathcal{Q}_{T+1}^k \equiv 0$. For $j=1,\ldots,M_T$, we solve the problem
\begin{equation}\label{backwardT}
\mathfrak{Q}_T ( x_{T-1}^k, \xi_{T j}  ) = 
\left\{ 
\begin{array}{l}
\displaystyle \inf_{x_T \in \mathbb{R}^n} c_{T j}^T x_T \\
A_{T j} x_{T} + B_{T j} x_{T-1}^k = b_{T j},
x_T \geq 0,
\end{array}
\right.
\mbox{ with dual }
\left\{ 
\begin{array}{l}
\sup_{\lambda} \lambda^T ( b_{T j} - B_{T j} x_{T-1}^k )\\
A_{T j}^T \lambda \leq c_{T j}.
\end{array}
\right.
\end{equation}
Let $\lambda_{T j}^k$ be an optimal solution of the dual problem above. We get
$$
\mathfrak{Q}_T ( x_{T-1}, \xi_{T j}  ) \geq \langle \lambda_{T j}^{k},  b_{T j} - B_{T j} x_{T-1} \rangle
$$
and compute $\theta_{T j}^k=b_{T j}^T \lambda_{T j}^{k}$ and  $\beta_{T j}^k = - B_{T j}^T \lambda_{T j}^k$.
Then for $t=T-1$ down to $t=2$, knowing $\mathcal{Q}_{t+1}^k \leq \mathcal{Q}_{t+1}$,
we solve  the problem below for $j=1,\ldots,M_t$,
\begin{equation}\label{backwardt}
{\underline{\mathfrak{Q}}}_t^k ( x_{t-1}^k, \xi_{t j}  ) = 
\left\{ 
\begin{array}{l}
\displaystyle \inf_{x_t} c_{t j}^T x_t + \mathcal{Q}_{t+1}^k ( x_t ) \\
x_t \in X_t( x_{t-1}^k , \xi_{t j} )
\end{array}
\right.
=
\left\{ 
\begin{array}{l}
\displaystyle \inf_{x_t, f} c_{t j}^T x_t + \sum_{\ell=1}^{M_{t+1}} p_{t+1 \ell}  f_{\ell} \\
A_{t j} x_{t} + B_{t j} x_{t-1}^k = b_{t j}, x_t \geq 0,\\
f_{\ell} \geq \theta_{t+1 \ell}^i + \langle \beta_{t+1 \ell}^i , x_t  \rangle, i=1,\ldots,k, \ell=1,\ldots,M_{t+1}.
\end{array}
\right.
\end{equation}
Observe that due to (H2) the above problem is feasible and has a finite optimal value. Therefore ${\underline{\mathfrak{Q}}}_t^k ( x_{t-1}^k, \xi_{t j}  )$
can be expressed as the optimal value of the corresponding dual problem
\begin{equation}\label{dualpbtback}
{\underline{\mathfrak{Q}}}_t^k ( x_{t-1}^k, \xi_{t j}  ) = 
\left\{
\begin{array}{l}
\displaystyle \sup_{\lambda, \mu} \lambda^T( b_{t j} - B_{t j} x_{t-1}^k   ) + \sum_{i=1}^k \sum_{\ell=1}^{M_{t+1}} \mu_{i \ell} \theta_{t+1 \ell}^i  \\
A_{t j}^T \lambda + \sum_{i=1}^k \sum_{\ell=1}^{M_{t+1}} \mu_{i \ell} \beta_{t+1 \ell}^i \leq c_{t j},\\
p_{t+1 \ell} = \sum_{i=1}^k \mu_{i \ell},\,\ell=1,\ldots,M_{t+1},\\
\mu_{i \ell} \geq 0,\,i=1,\ldots,k,\ell=1,\ldots,M_{t+1}.
\end{array}
\right.
\end{equation}
Let $(\lambda_{t j}^k, \mu_{t j}^k )$ an optimal solution of dual problem \eqref{dualpbtback}.
Using the fact that $\mathcal{Q}_{t+1}^k \leq \mathcal{Q}_{t+1}$, we get
$$
\mathfrak{Q}_t ( x_{t-1} , \xi_{t j}  ) \geq  {\underline{\mathfrak{Q}}}_t^k ( x_{t-1} , \xi_{t j}  ) \geq 
\langle \lambda_{t j}^k ,  b_{t j} - B_{t j} x_{t-1} \rangle + \langle \mu_{t j}^k , \theta_{t+1}^k \rangle
$$
and we compute 
$$
\theta_{t j}^k =\langle  \lambda_{t j}^k ,  b_{t j} \rangle +  \langle \mu_{t j}^k , \theta_{t+1}^k \rangle \mbox{ and }
\beta_{t j}^k =- B_{t j}^T \lambda_{t j}^k.
$$
In these expressions, vector $\theta_{t+1}^k$ has components $\theta_{t+1 \ell}^i, \ell=1,\ldots,M_{t+1}, i=1,\ldots,k$,
arranged in the same order as components $\mu_{t j}^k(\ell, i)$ of $\mu_{t j}^k$.\\

\par {\textbf{Step 4:} Do $k \leftarrow k+1$ and go to Step 2.

\subsection{Multicut stochastic decomposition with cut selection}\label{sec:cutselectionalgo}

We now describe a variant of MuDA that stores all cut coefficients $\theta_{t j}^i, \beta_{t j}^i$, 
and trial points $x_{t-1}^i$,
but that uses a reduced set of cuts $\mathcal{C}_{t j}^i$ to approximate functions $\mathfrak{Q}_{t}(\cdot,\xi_{t j})$
when solving problem  
\eqref{pbforwardpass} in the forward pass and
\eqref{backwardt} in the backward pass.
Let $\mathcal{S}_{t j}^k$ be the set of indices of the cuts selected at the end of iteration $k$ to approximate
$\mathfrak{Q}_{t}(\cdot,\xi_{t j})$.  In the end of the backward pass of iteration $k$, the variant of MuDA with cut selection 
computes  
approximations $\mathcal{Q}_t^k$ of $\mathcal{Q}_t$ given by 
\eqref{approxQt} now with $\mathfrak{Q}_t^k( \cdot, \xi_{t j})$ given by 
\begin{equation} \label{qtk}
\displaystyle \mathfrak{Q}_{t}^k( x_{t-1}, \xi_{t j} ) =  \max_{\ell \in \mathcal{S}_{t j}^k} \, \mathcal{C}_{t j}^\ell( x_{t-1}),
\end{equation}
where the set $\mathcal{S}_{t j}^k$ is a subset of the indices of the cuts that
have the largest value for at least one of the trial points computed so far. 
More precisely, sets $\mathcal{S}_{t j}^k$ are initialized taking $\mathcal{S}_{t j}^0=\{0\}$ and $\mathcal{S}_{t j}^1=\{1\}$.  
For $t \in \{2,\ldots,T\}$, sets $\mathcal{S}_{t j}^k$ are computed as follows.
For $i=1,\ldots,k$, $t=2,\ldots,T$, $j=1,\ldots,M_t$, let
$I_{t j}^{i k}$ be the set of cuts for $\mathfrak{Q}_{t}(\cdot,\xi_{t j})$ computed at iteration $k$ or before that have the largest value
at $x_{t-1}^i$:
\begin{equation} \label{indexlevel1}
I_{t j}^{i k} = \operatorname*{arg\,max}_{\ell=1,\ldots,k} \mathcal{C}_{t j}^{\ell}( x_{t-1}^i ),
\end{equation}
where the cut indices in $I_{t j}^{i k}$ are sorted in ascending order.
With a slight abuse of notation, we will denote the $\ell$-th smallest element in $I_{t j}^{i k}$ by $I_{t j}^{i k}(\ell)$.
For instance, if $I_{t j}^{i k}=\{2,30,50\}$ then $I_{t j}^{i k}(1)=2, I_{t j}^{i k}(2)=30$, and $I_{t j}^{i k}(3)=50$.
A cut selection strategy is given by a set of selectors $\mathcal{S}_{t j}(m), m=1,2,\ldots,$ where 
$\mathcal{S}_{t j}(m)$ is a subset of $\{1,2,\ldots,m\}$ giving the indices of the cuts to select in
$I_{t j}^{i k}$, through the relation
$$
\mathcal{S}_{t j}^k = \bigcup_{i=1}^k \left\{I_{t j}^{i k}(\ell) : \ell \in \mathcal{S}_{t j}(| I_{t j}^{i k} |) \right\},
$$
where $| I_{t j}^{i k} |$ is the cardinality of set $I_{t j}^{i k}$.
We require the selectors to satisfy the following assumption:
\begin{itemize}
\item[(H3)] for $t=2,\ldots,T$, $j=1,\ldots,M_t$, for every $m \geq 1$, $\mathcal{S}_{t j}(m) \subseteq \mathcal{S}_{t j}(m+1)$.
\end{itemize}
\begin{example}[Multicut Level 1 and Territory Algorithm] \label{excs1} The strategy $\mathcal{S}_{t j}(m)=\{1,2,\ldots,m\}$ selects all cuts that have the highest value for at least
one trial point. In the context of SDDP, this strategy was called Level 1 in (Philpott et al. 2012)  and
Territory Algorithm in (Pfeiffer et al. 2012). For this strategy, we have $\mathcal{S}_{T j}^k=\{1,\ldots,k\}$ for all $j$ and $k \geq 1$,
meaning that no cut selection is needed for the last stage $T$. This comes from the fact that for all  $k \geq 2$ and $1 \leq k_1 \leq k$,
cut $\mathcal{C}_{T j}^{k_1}$ is selected because it is one of the cuts with the highest value at $x_{T-1}^{k_1}$. Indeed,
for any $1 \leq k_2 \leq k$ with $k_2 \neq k_1$, 
since $\lambda_{T j}^{k_2}$ is feasible for problem \eqref{backwardT} with $x_{T-1}^{k}$ replaced by $x_{T-1}^{k_1}$, we
get
$$
\mathcal{C}_{T j}^{k_1}( x_{T-1}^{k_1} ) = \mathfrak{Q}_T( x_{T-1}^{k_1} , \xi_{T j} ) 
\geq \langle \lambda_{T j}^{k_2} , b_{T j} - B_{T j} x_{T-1}^{k_1} \rangle = \mathcal{C}_{T j}^{k_2}( x_{T-1}^{k_1} ).
$$
\end{example}
\begin{example}[Multicut Limited Memory Level 1] The strategy that eliminates the largest amount of cuts, 
called Multicut Limited Memory Level 1 (MLM Level 1 for short), consists in taking
a singleton for every set $\mathcal{S}_{t j}(m)$. For (H3) to be satisfied, this implies
$\mathcal{S}_{t j}(m)=\{1\}$. This choice corresponds to the Limited Memory Level 1 cut selection introduced in (Guigues 2017)
in the context of DDP. For that particular choice, at a given point, among the cuts that have the highest value
only the oldest (i.e., the cut that was first computed among the cuts that have the highest value at that point) is selected.
\end{example}
The computation of $\mathcal{S}_{t j}^k$, i.e., of the cut indices to select at iteration $k$, is performed in the backward pass
(immediately after computing cut $\mathcal{C}_{t j}^k$)
using the pseudo-code given in the left and right  panels of Figure \ref{figurecut1} for
 the Multicut Level 1 and MLM Level 1 cut selection strategies respectively.

In this pseudo-code, we use 
the notation $I_{t j}^i$ in place of $I_{t j}^{i k}$. We also store in variable
$m_{t j}^i$ the current value of the highest cut for $\mathfrak{Q}_{t}(\cdot, \xi_{t j})$
at $x_{t-1}^{i}$. At the end of the first iteration, we initialize
$m_{t j}^1 = \mathcal{C}_{t j}^1( x_{t-1}^{1})$. After cut $\mathcal{C}_{t j}^k$ is computed at 
iteration $k \geq 2$, these variables are updated using the relations
$$   
\left\{
\begin{array}{lll}
m_{t j}^i &  \leftarrow & \max (m_{t j}^i , \mathcal{C}_{t j}^k ( x_{t-1}^{i}) ),\;i=1,\ldots,k-1,\\
m_{t j}^k &  \leftarrow & \max (\mathcal{C}_{t j}^{\ell }( x_{t-1}^{k}), \ell=1,\ldots,k ).
\end{array}
\right. 
$$
We also use an array  of Boolean called {\tt{Selected}} using the information given by variables $I_{t j}^i$
whose $\ell$-th entry is {\tt{True}} if cut $\ell$ is selected for $\mathfrak{Q}_{t}(\cdot, \xi_{t j})$ and {\tt{False}} otherwise. 
This allows us  to avoid copies of cut indices that may appear in $I_{t {j}}^{i_1 k}$ and $I_{t {j}}^{i_2 k}$ with $i_1 \neq i_2$.
\begin{figure}
\begin{tabular}{|c|c|}
 \hline 
Multicut Level 1 & MLM Level 1 \\
\hline
\begin{tabular}{l}
$I_{t j}^{k}=\{k\}$, $m_{t  j}^k =\mathcal{C}_{t j}^k ( x_{t-1}^k )$.\\
{\textbf{For}} $\ell=1,\ldots,k-1$,\\
\hspace*{0.3cm}{\textbf{If }}$\mathcal{C}_{t j}^k( x_{t-1}^{\ell} ) > m_{t j}^{\ell}$ \\
\hspace*{0.6cm}$I_{t j}^{\ell}=\{k\},\; m_{t j}^{\ell}=\mathcal{C}_{t j}^k( x_{t-1}^{\ell} )$\\
\hspace*{0.3cm}{\textbf{Else if}} $\mathcal{C}_{t j}^k( x_{t-1}^{\ell} ) = m_{t j}^{\ell}$ \\
\hspace*{0.6cm}$I_{t j}^{\ell}=I_{t j}^{\ell} \cup \{k\}$\\
\hspace*{0.3cm}{\textbf{End If}}\\
\hspace*{0.3cm}{\textbf{If }}$\mathcal{C}_{t j}^{\ell}( x_{t-1}^{k} ) > m_{t j}^k$ \\
\hspace*{0.6cm}$I_{t j}^k =\{\ell\},\; m_{t j}^k = \mathcal{C}_{t j}^{\ell}( x_{t-1}^{k} )$\\
\hspace*{0.3cm}{\textbf{Else if}} $\mathcal{C}_{t j}^{\ell}( x_{t-1}^{k} ) = m_{t j}^{k}$ \\
\hspace*{0.6cm}$I_{t j}^{k}=I_{t j}^{k} \cup \{\ell\}$\\
\hspace*{0.3cm}{\textbf{End If}}\\
{\textbf{End For}}\\
{\textbf{For}} $\ell=1,\ldots,k$,\\
\hspace*{0.3cm}{\tt{Selected}}[$\ell$]={\tt{False}}\\
{\textbf{End For}}\\
{\textbf{For}} $\ell=1,\ldots,k$\\
\hspace*{0.3cm}{\textbf{For}} $m=1,\ldots,|I_{t j}^{\ell}|$\\
\hspace*{0.6cm}{\tt{Selected}}[$I_{t j}^{\ell}[m]$]={\tt{True}}\\
\hspace*{0.3cm}{\textbf{End For}}\\
{\textbf{End For}}\\
$\mathcal{S}_{t j}^k=\emptyset$\\
{\textbf{For}} $\ell=1,\ldots,k$\\
\hspace*{0.3cm}{\textbf{If}} {\tt{Selected}}[$\ell$]={\tt{True}} \\
\hspace*{0.6cm}$\mathcal{S}_{t j}^k=\mathcal{S}_{t j}^k \cup \{\ell\}$\\
\hspace*{0.3cm}{\textbf{End If}}\\
{\textbf{End For}}\\
\end{tabular}
&
\begin{tabular}{l}
\vspace*{-0.45cm}\\
$I_{t j}^{k}=\{1\}$, $m_{t  j}^k =\mathcal{C}_{t j}^1 ( x_{t-1}^k )$.\\
{\textbf{For}} $\ell=1,\ldots,k-1$,\\
\hspace*{0.3cm}{\textbf{If }}$\mathcal{C}_{t j}^k( x_{t-1}^{\ell} ) > m_{t j}^{\ell}$ \\
\hspace*{0.6cm}$I_{t j}^{\ell}=\{k\},\; m_{t j}^{\ell}=\mathcal{C}_{t j}^k( x_{t-1}^{\ell} )$\\
\hspace*{0.3cm}{\textbf{End If}}\\
\vspace*{0.4cm}\\
\hspace*{0.3cm}{\textbf{If }}$\mathcal{C}_{t j}^{\ell + 1}( x_{t-1}^{k} ) > m_{t j}^k$ \\
\hspace*{0.6cm}$I_{t j}^k =\{\ell+1\},\; m_{t j}^k = \mathcal{C}_{t j}^{\ell + 1}( x_{t-1}^{k} )$\\
\hspace*{0.3cm}{\textbf{End If}}\\
\vspace*{0.4cm}\\
{\textbf{End For}}\\
{\textbf{For}} $\ell=1,\ldots,k$,\\
\hspace*{0.3cm}{\tt{Selected}}[$\ell$]={\tt{False}}\\
{\textbf{End For}}\\
{\textbf{For}} $\ell=1,\ldots,k$\\
\hspace*{0.3cm}{\textbf{For}} $m=1,\ldots,|I_{t j}^{\ell}|$\\
\hspace*{0.6cm}{\tt{Selected}}[$I_{t j}^{\ell}[m]$]={\tt{True}}\\
\hspace*{0.3cm}{\textbf{End For}}\\
{\textbf{End For}}\\
$\mathcal{S}_{t j}^k=\emptyset$\\
{\textbf{For}} $\ell=1,\ldots,k$\\
\hspace*{0.3cm}{\textbf{If}} {\tt{Selected}}[$\ell$]={\tt{True}} \\
\hspace*{0.6cm}$\mathcal{S}_{t j}^k=\mathcal{S}_{t j}^k \cup \{\ell\}$\\
\hspace*{0.3cm}{\textbf{End If}}\\
{\textbf{End For}}\\
\end{tabular}
\\
\hline
\end{tabular}

\caption{Pseudo-codes for the computation of set $\mathcal{S}_{t j}^k$ for fixed $t\in \{2,\ldots,T\}, k \geq 2, j=1,\ldots,M_t$, and two cut selection strategies.}
\label{figurecut1}
\end{figure}

\section{Convergence analysis}\label{sec:convanalysis}

Theorem \ref{convproof} proves that CuSMuDA converges in a finite number of iterations.
We will make the following assumption:
\begin{itemize}
\item[(H4)] The samples in the forward passes are independent: $(\tilde \xi_2^k, \ldots, \tilde \xi_T^k)$ is a realization of
$\xi^k=(\xi_2^k, \ldots, \xi_T^k) \sim (\xi_2, \ldots, \xi_T)$ 
and $\xi^1, \xi^2,\ldots,$ are independent.
\end{itemize}
The convergence proof is based on the following lemma:
\begin{lemma} \label{lemmafinitecuts}
Assume that all subproblems in the forward and backward passes
of CuSMuDA
are solved using an algorithm
that necessarily outputs an extreme point of the feasible set (for instance the simplex algorithm).
Let assumptions (H1), (H2), (H3), and (H4) hold.
Then almost surely, there exists $k_0 \geq 1$ such that 
for every $k \geq k_0$, $t=2,\ldots,T$, $j=1,\ldots,M_t$, we have
\begin{equation} \label{ineqk0}
\mathfrak{Q}_{t}^k (\cdot, \xi_{t j})=\mathfrak{Q}_{t}^{k_0}(\cdot, \xi_{t j})
\mbox{ and }\mathcal{Q}_{t}^k =\mathcal{Q}_{t}^{k_0}.
\end{equation}
\end{lemma}
\begin{proof} 
Let $\Omega_1$ be the event on the sample space $\Omega$ of sequences of forward scenarios such that 
every scenario is sampled an infinite number of times.
By Assumption (H4), this event $\Omega_1$ has probability one.

Consider a realization $\omega \in \Omega$ of CuSMuDA in $\Omega_1$ corresponding to realizations $(\tilde \xi^k_{1:T})_k$ of  $(\xi^k_{1:T})_k$ in the forward pass. To simplify, we will drop $\omega$ in the notation. For instance,
we will simply write $x_t^k, \mathcal{Q}_t^k$  for realizations  $x_t^k( \omega)$ and $\mathcal{Q}_t^k(\cdot)(\omega)$
of $x_t^k, \mathcal{Q}_t^k$ given realization $\omega \in \Omega$ of CuSMuDA.

\if{
\paragraph{{\textbf{1$^0$}}} We show by induction that the set of possible trial points $\{x_{t}^k : k \geq 1 \}$ is finite for every $t=1,\ldots,T-1$.
\paragraph{{\textbf{1$^0.$a.}}} $x_1^k$ is an extreme point of the polyhedron $X_1(x_0, \xi_1)$ which
has a finite number of extreme points since $A_1$ has a finite number of rows and columns.
Therefore $x_1^k$ can only take a finite number of values.
\paragraph{{\textbf{1$^0.$b.}}} Assume now that $x_{t-1}^k$ can only take a finite number of values for some $t\in \{2,\ldots,T\}$.
Recall that $x_t^k$ is an extreme point of one of the sets $X_t(x_{t-1}^k, \xi_{t j})$ where $j \in \{1,\ldots,M_t\}$.
Since each matrix $A_{t j}$ has a finite number of
rows and columns, each set  $X_t(x_{t-1}^k, \xi_{t j})$ has a finite number of extreme points.
Using the induction hypothesis and $M_t<+\infty$, there is a finite number of these sets and therefore
$x_t^k$ can only take a finite number of  values. This shows {\textbf{1$^0$.}}
}\fi
\par  We show by induction on $t$ that the number of different cuts computed by the algorithm is finite and that
after some iteration $k_t$ the same cuts are selected for functions $\mathfrak{Q}_{t}(\cdot, \xi_{t j})$.
Our induction hypothesis $\mathcal{H}(t)$ for $t\in \{2,\ldots,T\}$ is
that the sets $\{(\theta_{t j}^k , \beta_{t j}^k) : k \in \mathbb{N}\}, j=1,\ldots,M_t$, are finite and there exists some
finite $k_t$ such that for every $k>k_t$ we have
\begin{equation}\label{inductionT}
\{(\theta_{t j}^{\ell} , \beta_{t j}^{\ell}) : \ell \in S_{t j}^k \}=
\{(\theta_{t j}^{\ell} , \beta_{t j}^{\ell}) : \ell \in S_{t j}^{k_t} \}
\mbox{ and }x_{t-1}^k = x_{t-1}^{k_t},
\end{equation}
for every $j=1,\ldots,M_t$.
We will denote by $\mathcal{I}_{t j}^{i k}$ the set $\left\{I_{t j}^{i k}(\ell) : \ell \in \mathcal{S}_{t j}(| I_{t j}^{i k} |) \right\}$.
We first show, in items {{\textbf{a}}} and {{\textbf{b}}}  below that $\mathcal{H}(T)$ holds.
\par {{\textbf{a.}}} Observe that $\lambda_{T j}^k$ defined in the backward pass of CuSMuDA is an extreme point of the polyhedron
$\{\lambda : A_{T j}^\top \lambda \leq c_{T j}\}$. 
This polyhedron is a finite intersection of closed half spaces in finite dimension (since
$A_{T j}$ has a finite number of rows and columns) and therefore has a finite number of extreme points. It follows that $\lambda_{T j}^k$ can only take a finite
number of values, same as $(\theta_{T j}^k, \beta_{T j}^k)=(\langle \lambda_{T j}^k, b_{T j} \rangle , -B_{T j}^\top \lambda_{T j}^k)$, and
there exists ${\bar k}_{T}$ such that for every $k>{\bar k}_T$ and every $j$, each cut $\mathcal{C}_{T j}^k$ is a copy of a cut $\mathcal{C}_{T j}^{k'}$ with $1 \leq k' \leq {\bar k}_T$ (no new
cut is computed for functions $\mathfrak{Q}_{T}(\cdot, \xi_{T j})$ for $k>{\bar k}_T$).

Now recall that $x_{T-1}^k$ computed in the forward pass is a solution of
\eqref{pbforwardpass} with $t=T-1$ and this optimization problem
can be written as a linear problem 
adding variables $f_1,f_2,\ldots,f_{M_T}$
replacing in the objective $\mathcal{Q}_{T}^{k-1}(x_{T-1})$
by $\sum_{\ell=1}^{M_T} p_{T \ell} f_{\ell}$ and adding the linear
constraints 
$f_{\ell} \geq \theta_{T \ell}^i + \langle \beta_{T \ell}^i , x_{T-1} \rangle$,
$i=1,\ldots,k-1,$ $\ell=1,\ldots,M_T$.
On top of that, for iterations $k>{\bar k}_T$, since functions 
$\mathfrak{Q}_T^k(\cdot,\xi_{T j})$ are made of a collection 
of cuts taken from
the finite and fixed
set of cuts $\mathcal{C}_{T j}^{\ell}, \ell \leq \bar k_T$, the set of 
possible functions 
$(\mathfrak{Q}_T^k(\cdot,\xi_{T j}))_{k \geq 1}$ and therefore of possible functions 
$(\mathcal{Q}_T^k)_{k \geq 1}$ is finite.
It follows that there is a finite set of possible polyhedrons for the feasible
set of \eqref{pbforwardpass} (with $t=T-1$) rewritten as a linear program as we have just
explained, adding variables $f_1,f_2,\ldots,f_{M_T}$.
Since these polyhedrons have a finite number of extreme points (recall that
there is a finite number of different linear constraints), there is only
a finite number of possible trial points $(x_{T-1}^k)_{k \geq 1}$.
Therefore we can assume without loss of generality that $\bar k_T$ is such that 
for iterations $k>\bar k_T$, trial point $x_{T-1}^k$ is also a copy of a trial point
$x_{T-1}^{k'}$ with $k' \leq \bar k_T$.

\par {\textbf{b}}. We show that for every $i \geq 1$ and $j=1,\ldots,M_T$,
there exists $1 \leq i' \leq {\bar k}_T$ 
and  $k_{T j}(i') \geq  {\bar k}_T$
such that for every $k \geq  \max(i, k_{T j}(i'))$ we have
\begin{equation}\label{inductioncruc}
\left\{ (\theta_{T j}^{\ell}, \beta_{T j}^{\ell}) :  \ell \in \mathcal{I}_{T j}^{i k} \right\}=
\left\{ (\theta_{T j}^{\ell}, \beta_{T j}^{\ell}) :  \ell \in \mathcal{I}_{T j}^{i' k_{T j}(i')} \right\},
\end{equation}
which will show $\mathcal{H}(T)$ with  $k_T = \max_{1 \leq  j \leq M_T, 1 \leq i' \leq {\bar k}_T} k_{T j}( i')$.
Let us show that \eqref{inductioncruc} indeed holds.
Let us take $i \geq 1$ and $j \in \{1,\ldots,M_T \}$.
If $1 \leq i \leq {\bar k}_{T}$ define $i'=i$. Otherwise
due to {\textbf{a.}} we can find $1 \leq i' \leq {\bar k}_T$ such that 
$x_{T-1}^i = x_{T-1}^{i'}$ which implies $I_{T j}^{i k}= I_{T j}^{i' k}$ for every $k \geq i$.
Now consider the sequence of sets $(I_{T j}^{i' k})_{k \geq  {\bar k}_T}$.
Due to the definition of $\bar k_T$, the sequence  $(|I_{T j}^{i' k}|)_{k > {\bar k}_T}$ is nondecreasing and therefore
two cases can happen:
\begin{itemize}
\item[(A)] there exists $k_{T j}(i')$ such that for $k \geq k_{T j}(i')$ we have $I_{T j}^{i' k}=I_{T j}^{i' k_{T j}(i')}$ (the cuts
computed after iteration $k_{T j}(i')$ are not active at $x_{T-1}^{i'}$). In this case 
$\mathcal{I}_{T j}^{i' k} = \mathcal{I}_{T j}^{i' k_{T j}(i')}$
for $k \geq k_{T j}(i')$, $\mathcal{I}_{T j}^{i k} = \mathcal{I}_{T j}^{i' k} = \mathcal{I}_{T j}^{i' k_{T j}(i')}$
for $k \geq \max( k_{T j}(i'), i)$ and 
\eqref{inductioncruc} holds.
\item[(B)] The sequence $(|I_{T j}^{i' k}|)_{k \geq  {\bar k}_T}$ is unbounded.
Due to Assumption (H3), the sequence $(|\mathcal{S}_{T j}(m)|)_m$ is nondecreasing.
If there exists $k_{T j}>\bar k_T$ such that  $\mathcal{S}_{T j}(k)=\mathcal{S}_{T j}(k_{T j})$ for $k \geq k_{T j}$
then if $k_{T j}( i' )$ is the smallest $k$ such that 
$|I_{T j}^{i' k}| \geq k_{T j}$ then for every $k \geq k_{T j}( i' )$ we have 
$\mathcal{I}_{T j}^{i' k} = \mathcal{I}_{T j}^{i' k_{T j}(i')}$ and
for $k \geq \max( k_{T j}(i'), i)$ we deduce 
$\mathcal{I}_{T j}^{i k} = \mathcal{I}_{T j}^{i' k_{T j}(i')}$
and \eqref{inductioncruc} holds.
Otherwise the sequence $(|\mathcal{S}_{T j}(m)|)_m$ is unbounded and an infinite number of cut indices
are selected from the sets $(I_{T j}^{i' k})_{k \geq i'}$ to make up sets $(\mathcal{I}_{T j}^{i' k})_{k \geq i'}$. 
However, since there is a finite number of different cuts, there is only a finite number of iterations where a new
cut can be selected from $I_{T j}^{i' k}$ and therefore there exists $k_{T j}(i')$ such that \eqref{inductioncruc} holds
for $k \geq \max( k_{T j}(i'), i)$.
\end{itemize}
\paragraph{{\textbf{c.}}}
Now assume that $\mathcal{H}(t+1)$ holds for some $t \in \{2,\ldots,T-1\}$. We want to show $\mathcal{H}(t)$.
Consider the set $\mathcal{D}_{t j k}$ of points of form
\begin{equation}\label{relationthetabeta}
(\langle  \lambda ,  b_{t j} \rangle +  \sum_{\ell=1}^{M_{t+1}} \sum_{i \in S_{t+1 \ell}^k} \mu_{i \ell} \theta_{t+1 \ell}^i , - B_{t j}^\top \lambda )
\end{equation}
 where $(\lambda, \mu )$ is an extreme point of 
the set $\mathcal{P}_{t j k}$ of points $(\lambda, \mu)$
satisfying 
\begin{equation}\label{extpointpoly}
\begin{array}{l}
\mu \geq 0, p_{t+1 \ell}=\sum_{i \in S_{t+1 \ell}^k} \mu_{i \ell},\;\ell=1,\ldots,M_{t+1}, \;A_{t j}^\top \lambda + \sum_{\ell =1}^{M_{t+1}} \sum_{i \in S_{t+1 \ell}^k} \mu_{i \ell} \beta_{t+1 \ell}^i \leq c_{t j}.
\end{array}
\end{equation}
We claim that for every $k>k_{t+1}$, every point from $\mathcal{D}_{t j k}$
can be written as a point from $\mathcal{D}_{t j k_{t+1}}$, i.e., a point
of form \eqref{relationthetabeta}  with $k$ replaced by $k_{t+1}$ and $(\lambda, \mu )$ an extreme point of 
the set $\mathcal{P}_{t j k_{t+1}}$. Indeed, take a point from $\mathcal{D}_{t j k}$, i.e., a point
of form \eqref{relationthetabeta}
with $(\lambda, \mu)$ an extreme point of $\mathcal{P}_{t j k}$ and $k>k_{t+1}$. 
It can be written as a point from $\mathcal{D}_{t j k_{t+1}}$, i.e., a point of form 
\eqref{relationthetabeta} with $k$ replaced by $k_{t+1}$ and $(\lambda, \hat \mu)$ in the place
of $(\lambda, \mu)$
where $(\lambda, \hat \mu)$ is an extreme point of $\mathcal{P}_{t j k_{t+1}}$  obtained 
replacing the basic columns $\beta_{t+1 \ell}^i$ with $i \in S_{t+1 \ell}^k, i \notin S_{t+1 \ell}^{k_{t+1}}$
associated with $\mu$
by columns $\beta_{t+1 \ell}^{i'}$ with $i' \in S_{t+1 \ell}^{k_{t+1}}$
such that $(\theta_{t+1 \ell}^{i}, \beta_{t+1 \ell}^{i})=(\theta_{t+1 \ell}^{i'}, \beta_{t+1 \ell}^{i'})$ (this is possible due to $\mathcal{H}(t+1)$).

Since $\mathcal{P}_{t j k_{t+1}}$ has a finite number of extreme points,
the set $\mathcal{D}_{t j k}$ has a finite cardinality and recalling that
for CuSMuDA $(\theta_{t j}^k, \beta_{t j}^k) \in \mathcal{D}_{t j k}$, the cut coefficients $(\theta_{t j}^k, \beta_{t j}^k)$
can only take a finite number of values. This shows the first part of $\mathcal{H}(t)$. 
Therefore, there exists ${\bar k}_{t}$ such that for every $k>{\bar k}_t$ and every $j$, each cut $\mathcal{C}_{t j}^k$ is a copy of a cut $\mathcal{C}_{t j}^{k'}$ with $1 \leq k' \leq {\bar k}_t$ (no new
cut is computed for functions $\mathfrak{Q}_{t}(\cdot, \xi_{t j})$ for $k>{\bar k}_t$).
As for the induction step $t=T$, this clearly implies that
after some iteration, no new trial points are computed and therefore we can assume without loss
of generality that $\bar k_t$ is such that 
for $k>\bar k_T$ trial point $x_{t-1}^k$ is a copy of some $x_{t-1}^{\ell}$ with $1 \leq \ell \leq {\bar k}_t$.

Finally, we can show \eqref{inductionT}
proceeding as in {\textbf{b.}}, replacing $T$ by $t$. This achieves the proof of $\mathcal{H}(t)$.

Gathering our observations, we have shown that \eqref{ineqk0} holds with $k_0 = \max_{t=2,\ldots,T} k_t$.$\hfill \qed$
\end{proof}
\begin{remark}
For Multicut Level 1 and MLM Level 1 cut selection strategies corresponding to selectors
$\mathcal{S}_{t j}$  satisfying respectively $\mathcal{S}_{t j}(m)=\{1,\ldots,m\}$
and $\mathcal{S}_{t j}(m)=\{1\}$, integers $k_0, {\tilde k}_t$ defined in
Lemma \ref{lemmafinitecuts} and its proof satisfy $k_0 \leq \max_{t=2,\ldots,T} {\tilde k}_t$.
For other selectors $\mathcal{S}_{t j}$ this relation is not necessarily satisfied (see {\textbf{b.}}-(B) of the proof
of the lemma).
\end{remark}

\begin{theorem}\label{convproof} Let Assumptions (H1), (H2), (H3), and (H4) hold.
Then Algorithm CuSMuDA converges with probability one in a finite number of iterations to a policy which is an optimal solution
of \eqref{firststodp}-\eqref{secondstodp}. 
\end{theorem}
\begin{proof} Let $\Omega_1$ be defined as in the proof of Lemma \ref{lemmafinitecuts}  and  
let $\Omega_2$ be the event such that $k_0$ defined in Lemma \ref{lemmafinitecuts} is finite.
Note that $\Omega_1 \cap \Omega_2$ has probability $1$.
Consider a realization of CuSMuDA in $\Omega_1 \cap \Omega_2$ corresponding to realizations
$(\tilde \xi^k_{1:T})_k$ of  $(\xi^k_{1:T})_k$ in the forward pass 
and let $(x_1^*, x_2^*( \cdot )$, $\ldots , x_T^*( \cdot ) )$
be the policy obtained from iteration $k_0$ on which uses recourse functions $\mathcal{Q}_{t+1}^{k_0}$ instead of $\mathcal{Q}_{t+1}$.
Recall that policy $(x_1^*, x_2^*( \cdot )$, $\ldots , x_T^*( \cdot ) )$ is optimal if for every realization
${\tilde \xi}_{1:T} := (\xi_1, {\tilde \xi}_2, \ldots,{\tilde \xi}_T)$ of $\xi_{1:T}:=(\xi_1, \xi_2,\ldots,\xi_T)$, we have that $x_t^* (\xi_1, {\tilde \xi}_2, \ldots,{\tilde \xi}_t )$
solves
\begin{equation}\label{optline}
\mathfrak{Q}_t ( x_{t-1}^{*}( {\tilde \xi}_{1:t-1} ), {\tilde \xi}_t   ) =
\displaystyle \inf_{x_t} \{ {\tilde c}_t^T x_t  + \mathcal{Q}_{t+1}( x_t) \;:\; x_t \in X_t( x_{t-1}^* ({\tilde \xi}_{1:t-1})  , {\tilde \xi}_t )    \}  
\end{equation}
for every $t=1,\ldots,T$, with the convention that $x_{0}^* = x_0$. 
We prove for $t=1,\ldots,T$, 
$$
\begin{array}{l}
{\overline{\mathcal{H}}}(t): \mbox{ for every }k \geq k_0 \mbox{ and for every sample }{\tilde \xi}_{1:t}=(\xi_1, {\tilde \xi}_2, \ldots,{\tilde \xi}_t) \mbox{ of }(\xi_1,\xi_2,\ldots,\xi_t), \mbox{ we have}\\
\hspace*{1.1cm}{\underline{\mathfrak{Q}}}_t^{k} ( x_{t-1}^{*}( {\tilde \xi}_{1:t-1} ) , {\tilde \xi}_{t}  )= \mathfrak{Q}_t ( x_{t-1}^{*}( {\tilde \xi}_{1:t-1} ) , {\tilde \xi}_t  ).
\end{array}
$$
We show ${\overline{\mathcal{H}}}(1),\ldots,{\overline{\mathcal{H}}}(T)$ by induction.
${\overline{\mathcal{H}}}(T)$ holds since 
${\underline{\mathfrak{Q}}}_T^{k} = \mathfrak{Q}_T$
for every $k$. 
Now assume that ${\overline{\mathcal{H}}}(t+1)$ holds for some $t \in \{1,\ldots,T-1\}$.
We want to show ${\overline{\mathcal{H}}}(t)$.
Take an arbitrary $k \geq k_0$ and a  sample  ${\tilde \xi}_{1:t-1}=(\xi_1,{\tilde \xi}_2, \ldots,{\tilde \xi}_{t-1})$ of
$(\xi_1,\xi_2,\ldots,\xi_{t-1})$. We have for every $j=1,\ldots,M_t$, that
\begin{equation}\label{eqQtapprox}
\begin{array}{lll}
{\underline{\mathfrak{Q}}}_t^{k} ( x_{t-1}^{*}( {\tilde \xi}_{1:t-1} ) , \xi_{t j}  )&=&
\left\{
\begin{array}{l}
\inf \;c_{t j}^T x_t  + \mathcal{Q}_{t+1}^k ( x_t )\\
x_t \in X_t( x_{t-1}^{*}( {\tilde \xi}_{1:t-1} ), \xi_{t j})
\end{array}
\right.\\
&=& c_{t j}^T x_{t}^{*}( {\tilde \xi}_{1:t-1}, \xi_{t j} ) + \mathcal{Q}_{t+1}^{k} ( x_{t}^{*}( {\tilde \xi}_{1:t-1}, \xi_{t j} ) ).
\end{array}
\end{equation}
Now we check that for $j=1,\ldots,M_t$, we have
\begin{equation}\label{indcstep}
\mathcal{Q}_{t+1}^{k} ( x_{t}^{*}( {\tilde \xi}_{1:t-1}, \xi_{t j} ) ) = \mathcal{Q}_{t+1}  ( x_{t}^{*}( {\tilde \xi}_{1:t-1}, \xi_{t j} ) ).
\end{equation}
Indeed, if this relation did not hold, since 
$\mathcal{Q}_{t+1}^{k}  \leq \mathcal{Q}_{t+1} $, we would have 
$$
\mathcal{Q}_{t+1}^{k_0} ( x_{t}^{*}( {\tilde \xi}_{1:t-1}, \xi_{t j} ) ) =\mathcal{Q}_{t+1}^{k} ( x_{t}^{*}( {\tilde \xi}_{1:t-1}, \xi_{t j} ) ) < \mathcal{Q}_{t+1}  ( x_{t}^{*}( {\tilde \xi}_{1:t-1}, \xi_{t j} ) ).
$$
From the definitions of $\mathcal{Q}_{t+1}^k, \mathcal{Q}_{t+1}$, there exists $m \in \{1,\ldots,M_{t+1}\}$ such that
$$
\mathfrak{Q}_{t+1}^{k_0}( x_{t}^{*}( {\tilde \xi}_{1:t-1}, \xi_{t j} )  , \xi_{t+1 m})<\mathfrak{Q}_{t+1}( x_{t}^{*}( {\tilde \xi}_{1:t-1}, \xi_{t j} )  , \xi_{t+1 m}).
$$
Since the realization of CuSMuDA is in $\Omega_1$, there exists an infinite set of iterations such that the sampled scenario for stages 
$1,\ldots,t$, is $(\tilde \xi_{1:t-1}, \xi_{t j})$. Let $\ell$ be one of these iterations strictly greater than $k_0$.
Using ${\overline{\mathcal{H}}}(t+1)$, we have that 
$$
{\underline{\mathfrak{Q}}}_{t+1}^{\ell}( x_{t}^{*}( {\tilde \xi}_{1:t-1}, \xi_{t j} )  , \xi_{t+1 m}) =\mathfrak{Q}_{t+1}( x_{t}^{*}( {\tilde \xi}_{1:t-1}, \xi_{t j} )  , \xi_{t+1 m})
$$
which yields
$$
\mathfrak{Q}_{t+1}^{k_0}( x_{t}^{*}( {\tilde \xi}_{1:t-1}, \xi_{t j} )  , \xi_{t+1 m})<{\underline{\mathfrak{Q}}}_{t+1}^{\ell}( x_{t}^{*}( {\tilde \xi}_{1:t-1}, \xi_{t j} )  , \xi_{t+1 m})
$$
and at iteration $\ell>k_0$ we would construct a cut for $\mathfrak{Q}_{t+1}(\cdot, \xi_{t+1 m})$ at $x_{t}^{*}( {\tilde \xi}_{1:t-1}, \xi_{t j} )=x_t^{\ell}$
with value 
$${\underline{\mathfrak{Q}}}_{t+1}^{\ell}( x_{t}^{*}( {\tilde \xi}_{1:t-1}, \xi_{t j} )  , \xi_{t+1 m})$$ strictly larger than the value
at this point of all cuts computed up to iteration $k_0$.
Due to Lemma \ref{lemmafinitecuts}, this is not possible. Therefore, \eqref{indcstep} holds, which, plugged into \eqref{eqQtapprox}, gives
$$
{\underline{\mathfrak{Q}}}_t^{k} ( x_{t-1}^{*}( {\tilde \xi}_{1:t-1} ) , \xi_{t j}  )=
c_{t j}^T x_{t}^{*}( {\tilde \xi}_{1:t-1}, \xi_{t j} ) + \mathcal{Q}_{t+1} ( x_{t}^{*}( {\tilde \xi}_{1:t-1}, \xi_{t j} ) )
\geq \mathfrak{Q}_t ( x_{t-1}^{*}( {\tilde \xi}_{1:t-1} ) , \xi_{t j}  )
$$
(recall that $x_{t}^{*}( {\tilde \xi}_{1:t-1}, \xi_{t j} ) \in X_t( x_{t-1}^{*}( {\tilde \xi}_{1:t-1}),   \xi_{t j})$).
Since ${\underline{\mathfrak{Q}}}_t^{k} \leq \mathfrak{Q}_t$, we have shown
$ {\underline{\mathfrak{Q}}}_t^{k} ( x_{t-1}^{*}( {\tilde \xi}_{1:t-1} ) , \xi_{t j}  )= \mathfrak{Q}_t ( x_{t-1}^{*}( {\tilde \xi}_{1:t-1} ) , \xi_{t j}  )$
for every $j=1,\ldots,M_t$, which is ${\overline{\mathcal{H}}}(t)$. 
Therefore, we have proved that for $t=1,\ldots,T,$ for every realization $(\tilde \xi_{1:t-1}, \xi_{t j})$ of $\xi_{1:t}$, $x_{t}^{*}( {\tilde \xi}_{1:t-1}, \xi_{t j} )$
satisfies 
$c_{t j}^T  x_{t}^{*}( {\tilde \xi}_{1:t-1}, \xi_{t j} ) + \mathcal{Q}_{t+1}( x_{t}^{*}( {\tilde \xi}_{1:t-1}, \xi_{t j} ) ) = \mathfrak{Q}_t ( x_{t-1}^{*}( {\tilde \xi}_{1:t-1} ) , \xi_{t j}  )$,
meaning that for every $j=1,\ldots,M_t$,
$x_{t}^{*}( {\tilde \xi}_{1:t-1}, \xi_{t j} )$
is an optimal solution of \eqref{optline} written with $\tilde \xi_{1:t} = (\tilde \xi_{1:t-1}, \xi_{t j})$ and completes the proof.
\end{proof}
\begin{remark} The convergence proof above also shows the almost sure convergence in a finite
number of iterations for cut selection strategies that would always select more cuts
than any cut selection strategy satisfying Assumption (H3).
It also shows the convergence of Level $H$ cut selection from (Philpott et al. 2012) which
keeps the $H$ cuts having the largest values at each trial point.
\end{remark}

\section{Application to portfolio selection and inventory management}\label{sec:numsim}

\subsection{Portfolio selection} \label{portfolio}

\subsubsection{Model} \label{portfoliomodel}

We consider a portfolio selection problem with direct transaction costs over a discretized horizon
of $T$ stages. The direct buying and selling  transaction costs are
proportional to the amount of the transaction 
(Ben-Tal et al. 2000, Best and Hlouskova 2007).\footnote{This portfolio problem was solved using SDDP and SREDA (Stochastic REgularized Decomposition
Algorithm) in (Guigues et al. 2017).} 
Let $x_t( i )$ be the dollar value of asset $i=1,\ldots,n+1$ at the end of stage $t=1,\ldots,T$,
where asset $n+1$ is cash; $\xi_t(i)$ is the return of asset $i$ at $t$; 
$y_t(i)$ is the amount of asset $i$ sold at the end of $t$; 
$z_t(i)$ is the amount of asset $i$ bought at the end of $t$, 
$\eta_t(i) > 0$ and $\nu_t(i) > 0$ are respectively the proportional selling and buying transaction costs at $t$.
Each component $x_0(i),i=1,\ldots,n+1$, of $x_0$ is known.
The budget available at the beginning of the investment period is
$\sum_{i=1}^{n+1} \xi_{1}(i) x_{0}(i)$ and
$u(i)$ represents the maximal amount that can be invested in asset $i$. 

For $t=1,\ldots,T$, given a portfolio $x_{t-1}=(x_{t-1}(1),\ldots,x_{t-1}(n), x_{t-1}( n+1) )$ and $\xi_t$, we define the set $X_t(x_{t-1}, \xi_t)$
as the set of $(x_t, y_t, z_t) \in \mathbb{R}^{n+1} \small{\times} \mathbb{R}^{n} \small{\times} \mathbb{R}^{n}$ satisfying
\begin{equation}
x_t( n+1 ) = \xi_{t}( n+1 ) x_{t-1}( n+1 )   +\sum\limits_{i=1}^{n} \Big((1-\eta_t( i) )y_t( i )-  (1+\nu_t ( i ))z_t( i )\Big), \label{p1_3}
\end{equation}
and for $i=1,\ldots,n$,

\begin{subequations}\label{MI-CONS0}
\begin{align}
x_{t}(i)&= \xi_{t}(i) x_{t-1}(i)-y_t(i) +z_t(i),  \label{p1_2}\\
y_t(i) &\leq  \xi_{t}(i) x_{t-1}(i),  \label{p1_4}\\
x_t(i) &\leq u(i)  \sum\limits_{j=1}^{n+1} \xi_{t}(j) x_{t-1}(j),  \label{p1_5}\\
x_t(i) &\ge 0, y_t(i) \geq 0, z_t(i) \geq 0.    \label{p1_6}
\end{align}
\end{subequations}
Constraints \eqref{p1_3} are the cash flow balance constraints and define how much cash is available at each stage.
Constraints \eqref{p1_2} define the amount of security $i$ held at each stage $t$ and take into account the proportional transaction costs.
Constraints \eqref{p1_4} preclude selling an amount larger than the amount held. 
Constraints \eqref{p1_5} prevent the position in security $i$ at time $t$ from exceeding a proportion $u(i)$.
Constraints \eqref{p1_6} prevent short-selling and enforce the non-negativity of the amounts bought and sold. 

With this notation, the following dynamic programming equations of a risk-neutral portfolio model can be written:
for $t=T$, setting $\mathcal{Q}_{T+1}( x_T ) = \mathbb{E}[ \sum\limits_{i=1}^{n+1} \xi_{T+1}(i) x_{T}(i) ]$
we solve the problem
\begin{equation}\label{eq1dp}
\mathfrak{Q}_T \left( x_{T-1}, \xi_T \right)=
\left\{
\begin{array}{l}
\sup \; \mathcal{Q}_{T+1}( x_T )  \\
(x_T, y_T, z_T) \in X_T(x_{T-1}, \xi_T),
\end{array}
\right.
\end{equation}
while at stage $t=T-1,\dots,1$, we solve
\begin{equation}\label{eq2dp}
\mathfrak{Q}_t\left( x_{t-1}, \xi_t  \right)=
\left\{
\begin{array}{l}
\sup  \;  Q_{t+1}\left( x_{t} \right) \\
(x_t, y_t, z_t) \in X_t(x_{t-1}, \xi_t) , 
\end{array}
\right.
\end{equation}
where for $t=2,\ldots,T$, $\mathcal{Q}_t(x_{t-1})=\mathbb{E}[ \mathfrak{Q}_t\left( x_{t-1}, \xi_t  \right) ]$.
With this model, we maximize the expected return of the portfolio taking into account the transaction costs,
non-negativity constraints, and bounds imposed on the different securities.

\subsubsection{CuSMuDA for portfolio selection}

We assume that the return process $(\xi_t)$ satisfies Assumption (H1).\footnote{It is possible (at the expense of the
computational time) to incorporate stagewise dependant returns within the decomposition algorithms under consideration, for instance
including in the state vector the relevant history of the returns as in (Infanger and Morton 1996,  Guigues 2014).}
In this setting, we can solve the portfolio problem under consideration using
MuDA and CuSMuDA. For the sake of completeness, we show how to apply MuDA to this problem,
including the stopping criterion (CuSMuDA follows, incorporating one of the pseudo-codes from Figure \ref{figurecut1}). 
In this implementation, $N$ independent scenarios ${\tilde \xi}^k, k=(i-1)N+1,\ldots,iN$, of $(\xi_1,\xi_2,\ldots,\xi_T)$
are sampled in the forward pass of iteration $i$ to obtain $N$ sets of trial points (note that the convergence proof of Theorem \ref{convproof}
still applies for this variant of CuSMuDA). At the end of iteration $i$, we end up with approximate functions
$\mathfrak{Q}_t^i (x_{t-1} , \xi_{t j} ) = \max_{1 \leq \ell \leq iN} \langle \beta_{t j}^\ell ,  x_{t-1} \rangle$ for
$\mathfrak{Q}_t(\cdot, \xi_{t j})$  (observe that for this problem the cuts have no intercept). \\

\par {\textbf{Forward pass of iteration $i$.}} We generate $N$ scenarios 
${\tilde \xi}^k = (\tilde \xi_1^k,, \tilde \xi_2^k,\ldots,\tilde \xi_T^k), k=(i-1)N+1,\ldots,iN$, of $(\xi_2,\ldots,\xi_T)$
and solve for $k=(i-1)N+1,\ldots,iN$, and $t=1,\ldots,T-1$, the problem
$$
\left\{
\begin{array}{l}
\displaystyle \inf_{x_t, f} \; \sum_{\ell =1}^{M_{t+1}} p_{t+1 \ell} f_{\ell}\\
x_t \in X_t ( x_{t-1}^k , {\tilde \xi}_{t}^k  ),\\
f_{\ell} \geq \langle \beta_{t+1 \ell}^{m} ,  x _t\rangle, m=1,\ldots,(i-1)N, \ell =1, \ldots, M_{t+1},
\end{array}
\right.
$$
starting from $(x_{0}^k , {\tilde \xi}_1^k ) = (x_{0} , \xi_1 )$.\footnote{We use minimization instead of maximization subproblems. In this context, the optimal mean income is the opposite of the optimal value of the first stage problem.}
Let $x_{t}^k$ be an optimal solution.
For $t=T$ and $k=(i-1)N+1,\ldots,iN$, we solve 
$$
\left\{
\begin{array}{l}
\inf  \;- x_T^T  \mathbb{E}[\xi_{T+1}]  \\
x_T \in X_{T}( x_{T-1}^k , {\tilde \xi}_{T}^k  ),
\end{array}
\right.
$$
with optimal solution $x_T^k$.

In the end of the forward pass we compute 
the empirical mean ${\overline{\tt{Cost}}}^i$ and standard deviation $\sigma^i$ of the
cost on the sampled scenarios of iteration $i$:
\begin{equation} \label{formulastopping}
{\overline{\tt{Cost}}}^i = -\frac{1}{N}\sum_{k=(i-1)N + 1}^{iN} \mathbb{E}[\xi_{T+1}]^T  x_T^k, \;\;\sigma^i = \sqrt{\frac{1}{N} \sum_{k=(i-1)N+1}^{iN} \Big(-\mathbb{E}[\xi_{T+1}]^T  x_T^k  - {\overline{\tt{Cost}}}^i \Big)^2}. 
\end{equation}
This allows us to compute the upper end $z_{\sup}^i$ of a one-sided confidence interval on the
mean cost of the policy obtained at iteration $i$ given by
\begin{equation}\label{formulazsup}
z_{\sup}^i = {\overline{\tt{Cost}}}^i + \frac{\sigma^i}{\sqrt{N}} \Phi^{-1}(1-\alpha)
\end{equation}
where $\Phi^{-1}(1-\alpha)$ is the $(1-\alpha)$-quantile of the standard Gaussian distribution.\\
\par {\textbf{Backward pass of iteration $i$.}} For $k=(i-1)N+1,\ldots,iN$, we solve for $t=T$, $j=1,\ldots,M_T$,
\begin{equation}\label{back1equ}
\left\{
\begin{array}{l}
\inf  \;- x_T^T \mathbb{E}[\xi_{T+1}] \\
x_T \in X_T ( x_{T-1}^k , \xi_{T j}  )
\end{array}
\right.
\end{equation}
and for $t=T-1,\ldots,2$, $j=1,\ldots,M_t$,
\begin{equation}\label{back2equ}
\left\{
\begin{array}{l}
\displaystyle \inf_{x_t, f}  \; \sum_{\ell=1}^{M_{t+1}} p_{t+1 \ell} f_{\ell} \\
x_t \in X_t( x_{t-1}^k , \xi_{t j}  ),\\
f_{\ell} \geq \langle \beta_{t+1 \ell}^{m} ,  x _t \rangle, m=1,\ldots,iN, \ell=1,\ldots,M_{t+1}.
\end{array}
\right.
\end{equation}
For stage $t$ problem above with realization $\xi_{t j}$ of $\xi_t$ (problem \eqref{back1equ} for $t=T$ and \eqref{back2equ} for $t<T$), let
$\lambda_{t j}^k$ be the optimal Lagrange multipliers associated to the equality constraints
and let $\mu_{1 t j}^k \geq 0, \mu_{2 t j}^k \geq 0$ be the optimal Lagrange multipliers associated with respectively
the constraints $y_t(i) \leq  \xi_{t j}(i) x_{t-1}(i)$ and 
$x_t(i) \leq u(i)  \sum\limits_{\ell=1}^{n+1} \xi_{t j}(\ell) x_{t-1}(\ell)$. We compute
$$
\beta_{t j}^k = \Big( \lambda_{t j}^k - \langle u, \mu_{2 t j}^k \rangle {\textbf{e}} - \left[ \begin{array}{c}\mu_{1 t j}^k\\0  \end{array}   \right]  \Big) \circ \xi_{t j},
$$
where ${\textbf{e}}$ is a vector in $\mathbb{R}^{n+1}$ of ones and 
where for vectors $x, y$, the vector $x \circ y$ has components $(x \circ y)(i)=x(i) y(i)$.\\  
\par {\textbf{Stopping criterion.}}
In the end of the backward pass of iteration $i$, we solve
$$
\left\{
\begin{array}{l}
\displaystyle \inf_{x_1, f}  \;  \sum_{\ell=1}^{M_2} p_{2 \ell} f_{\ell}\\
x_1 \in X_1( x_{0} , \xi_1  ),\\
f_{\ell} \geq \langle \beta_{2 \ell}^{m} ,  x _1\rangle, m=1,\ldots,iN, \ell=1,\ldots,M_2,
\end{array}
\right.
$$
whose optimal value provides a lower bound $z_{\inf}^i$ on the optimal value
$\mathcal{Q}_1( x_0 )$ of the problem. 
Given a tolerance $\varepsilon>0$, 
the algorithm stops either when $z_{\inf}^i=0$ and $z_{\sup}^i \leq \varepsilon$ or when
\begin{equation}\label{stoppingcriterion} 
\left| z_{\sup}^i - z_{\inf}^i   \right| \leq \varepsilon \max(1, |z_{\sup}^i |  ).
\end{equation}
In the expression above, we use  $\varepsilon \max(1, |z_{\sup}^i |  )$ instead of $\varepsilon |z_{\sup}^i |$ in the right hand side to account for the case
$z_{\sup}^i=0$.

\subsubsection{Numerical results}

We compare three methods to solve the porfolio problem presented in the previous section:
MuDA with sampling that we have just described (denoted by {\tt{MuDA}} for short), CuSMuDA with Multicut Level 1 cut selection (denoted by {\tt{CuSMuDA CS 1}} for short) and 
CuSMuDA with MLM Level 1 cut selection  
(denoted by {\tt{CuSMuDA CS 2}} for short). 
The implementation was done in Matlab run on a laptop with Intel(R) Core(TM) i7-4510U CPU @ 2.00GHz.
All subproblems in the forward and backward passes were solved numerically using Mosek Optimization Toolbox (Andersen and Andersen 2013).

We fix
$u(i)=1, i=1,\ldots,n$, while $x_0$ has components uniformly distributed in $[0,10]$.
For the stopping criterion, we use $N=200, \alpha=0.025$, and
$\varepsilon=0.05$ in \eqref{stoppingcriterion}.
Below, we generate various instances of the portfolio problem as follows.
For fixed $T$ (number of stages [days for our experiment]) and $n$ (number of risky assets),
the distributions of $\xi_t(1:n)$ have $M_t=M$ realizations 
with $p_{t i}=\mathbb{P}(\xi_t = \xi_{t i})=1/M$, and
$\xi_1(1:n), \xi_{t 1}(1:n), \ldots, \xi_{t M}(1:n)$ chosen randomly among historical data of
daily returns of $n$ of the assets of the S\&P 500 index for the period 18/5/2009-28/5/2015.
For fixed $n$, the $n$ stocks chosen are given in Table \ref{liststocks}.\footnote{They correspond 
to the first $n$ stocks listed in our matrix of stock prices downloaded from Wharton Research Data Services (WRDS: {\url{https://wrds-web.wharton.upenn.edu/wrds/}}).}
The daily return $\xi_t(n+1)$ of the risk-free asset is $0.1$\% for all $t$.

\begin{table}
\begin{tabular}{|c|c|}
\hline 
Number of stocks & Stocks\\
\hline
$n=4$ & \begin{tabular}{l}  AAPL (Apple Inc.), XOM (Exxon Mobil Corp.),\\MSFT (Microsoft Corp.), JNJ (Johnson \& Johnson).\end{tabular}\\
\hline
$n=5$ &  AAPL, XOM, MSFT, JNJ, WFC (Wells Fargo).\\
\hline
$n=6$ &   AAPL, XOM, MSFT, JNJ, WFC, GE (General Electric)\\
\hline
\end{tabular}
re\caption{Subset of stocks chosen.}\label{liststocks}
\end{table}

Transaction costs are assumed to be known with
$\nu_t(i)=\mu_t(i)$
obtained sampling from the distribution of the random variable 
$0.08+0.06\cos(\frac{2\pi}{T} U_T )$ where $U_T$ is a random variable
with a discrete distribution over the set of integers $\{1,2,\ldots,T\}$. For our experiments the values $(5,4), (5,5), (5,6), (8,4), (8,5)$, and $(8,6)$ of the pair $(T,n)$ are considered.\\
  
\par {\textbf{Checking the implementations.}}
If all algorithms are correctly implemented the lower and upper bounds 
$z_{\inf}^K$ and $z_{\sup}^K$ computed at the last iteration $K$ of the algorithms
should be close to the optimal value $\mathcal{Q}_1( x_0)$ of the problem.
If we do not know the optimal value $\mathcal{Q}_1( x_0)$, we
can however check on Figure \ref{fig:f_1} that
the bounds for the 3 algorithms are very close to each
other at the last iteration. In Figure \ref{fig:f_1} we also observe, as expected,
that the lower bounds increase and the upper bounds 
tend to decrease along iterations. Note that few iterations are needed
to satisfy the stopping criterion for all instances. However, recall that at each iteration,
{\tt{MuDA}} computes $N M=200 \small{\times}20=4\,000$ cuts for each stage $t=2,\ldots,T$, meaning that, for instance,
for $T=5, n=6$, where $5$ iterations are needed for {\tt{MuDA}} to converge (see Figure \ref{fig:f_1}), $5(T-1)MN=80\,000$ cuts are
computed by {\tt{MuDA}}.\\
\begin{figure}
\centering
\begin{tabular}{cc}
\includegraphics[scale=0.6]{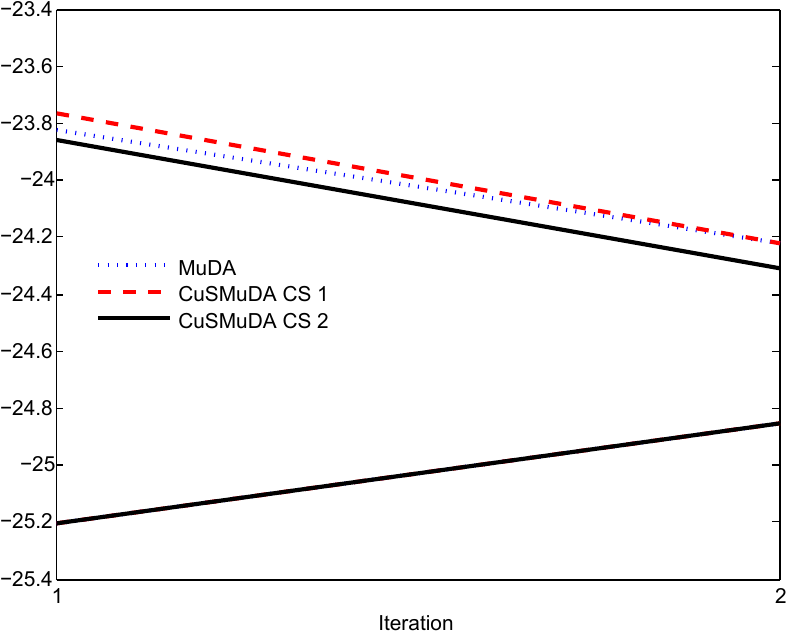}
&
\includegraphics[scale=0.6]{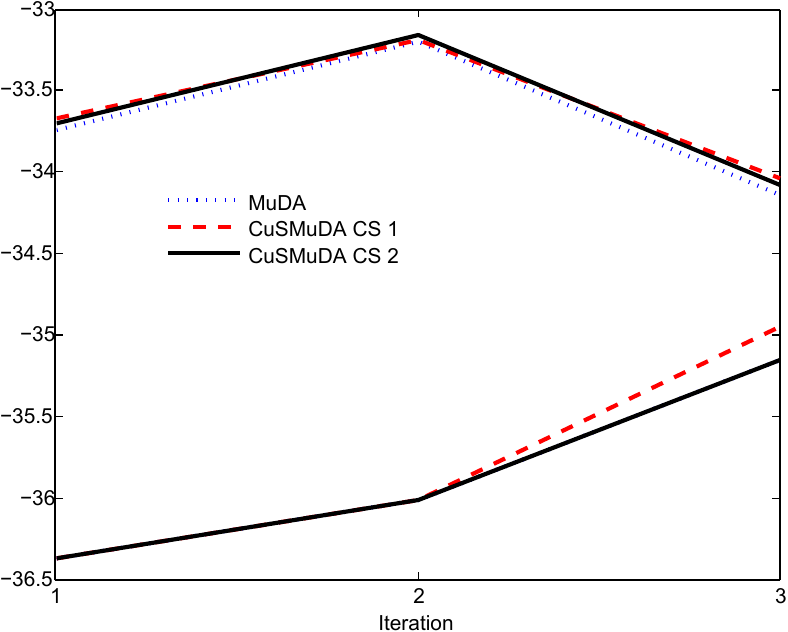}
\\
{$T=5, n=4, M=20$} & {$T=5, n=5, M=20$}\\
\includegraphics[scale=0.6]{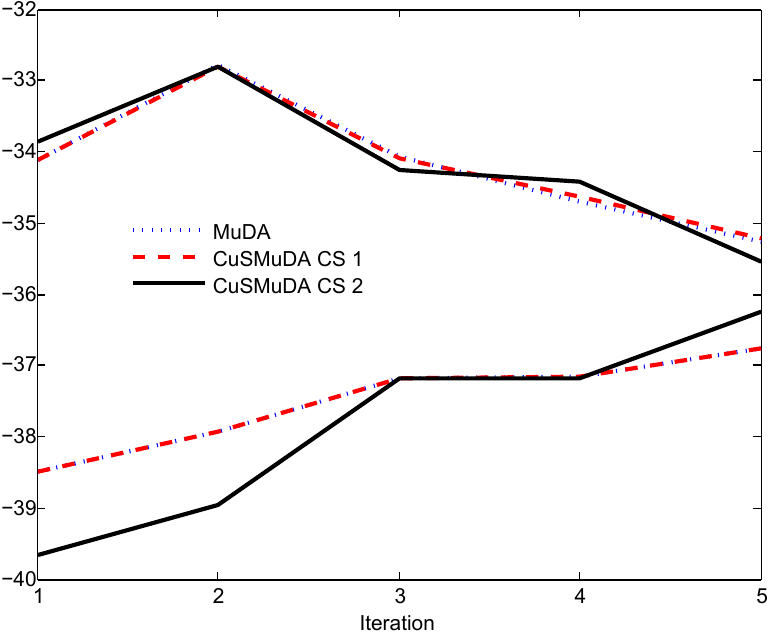}
&
\includegraphics[scale=0.6]{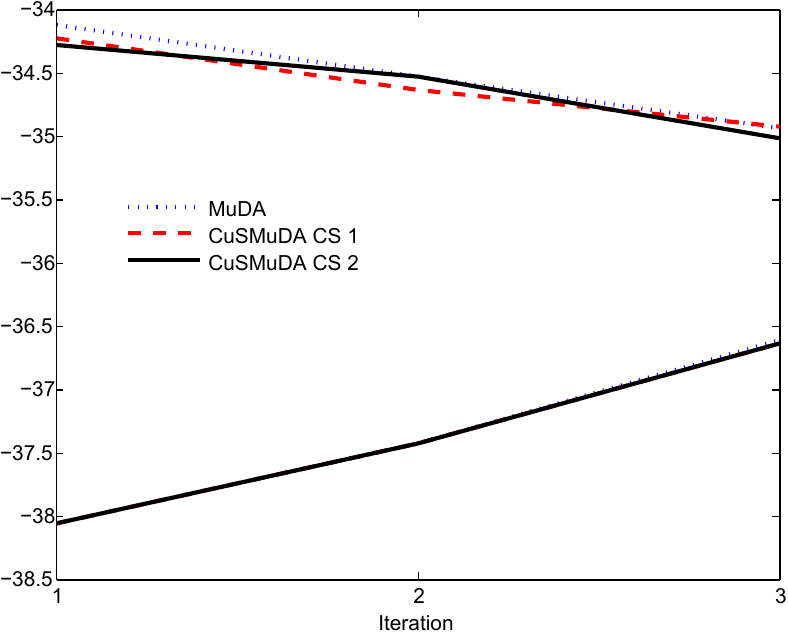}
\\
{$T=5, n=6, M=20$} &  {$T=8, n=4, M=10$}\\
\includegraphics[scale=0.6]{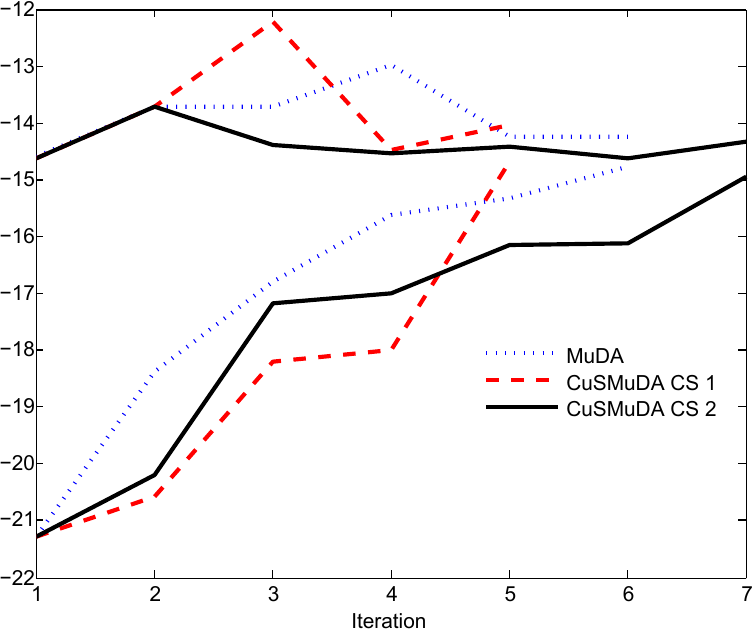}
&
\includegraphics[scale=0.6]{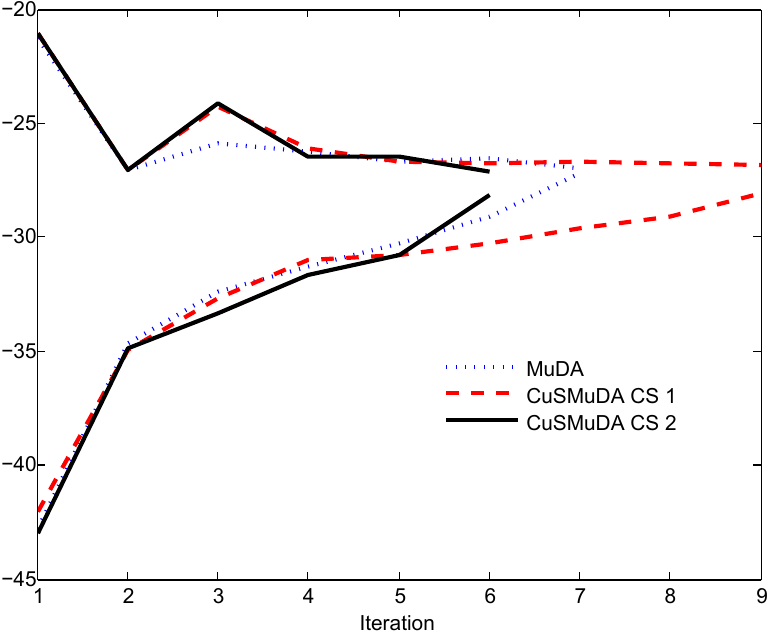}
\\
{$T=8, n=5, M=10$} & {$T=8, n=6, M=10$}\\
\end{tabular}
\caption{ \label{fig:f_1} Evolution of the upper bounds $z_{\sup}^i$ and lower bounds $z_{\inf}^i$
along the iterations of the algorithms.}
\end{figure}

\par {\textbf{Computational time.}} The computational time for solving 
several instances of the portfolio problem is given in Table \ref{tablerunningtime1}.
On these runs, we observe a huge variation in the computational time.
On all these runs, {\tt{CuSMuDA CS 2}} is much quicker than both {\tt{MuDA}} (between 5.1 and 12.6 times quicker)
and {\tt{CuSMuDA CS 1}} (between 10.3 and 21.9 times quicker).
To understand these results, let us analyze the proportion of cuts selected
by the two cut selection strategies.\\

\begin{table}
\centering
\begin{tabular}{|c||c|c|c|}
\hline
& {\tt{MuDA}} & {\tt{CuSMuDA CS 1}}  & {\tt{CuSMuDA CS 2}} \\
\hline
$T=5, n=4, M=20$ & 8.97  & 13.45     &  1.31  \\
\hline
$T=5, n=5, M=20$ & 17.03  &   29.41   &  2.38 \\
\hline
$T=5, n=6, M=20$ & 45.87  &   64.22   &  5.25 \\
\hline
$T=8, n=4, M=10$ &  7.94  &  16.27    &  1.56 \\
\hline
$T=8, n=5, M=10$ &  43.39 &  73.98    &  4.30 \\
\hline
$T=8, n=6, M=10$ & 63.47  &   110.55   &  5.05 \\
\hline
\end{tabular}
\caption{Computational time (in minutes) for solving instances of the portfolio problem of Section \ref{portfoliomodel} with {\tt{MuDA}}, {\tt{CuSMuDA CS 1}}, and {\tt{CuSMuDA CS 2}}.}
\label{tablerunningtime1}
\end{table}

\par {\textbf{Proportion of cuts selected.}} An important aspect of the cut selection strategies that directly impacts the
computational time is the proportion of cuts selected along the iterations of the algorithm. 

Selecting a small number of cuts, especially when many cuts have been computed, allows us to solve the problems in the backward and forward passes
much more quickly and can yield an important reduction in CPU time.
This is the case of {\tt{CuSMuDA CS 2}} which surprisingly succeeds in solving the problem selecting at each stage and
iteration a very small number of cuts, as can be seen in Figures \ref{fig:f_2}-\ref{fig:f_4}.
More precisely, the mean proportion of cuts selected (across iterations) is below 5\%
for all stages with {\tt{CuSMuDA CS 2}} (see Figure \ref{fig:f_2}). This means that a large number of cuts have the same value at some trial points and for each  such
trial point if {\tt{CuSMuDA CS 1}} selects all these cuts {\tt{CuSMuDA CS 2}} only selects the oldest of these cuts.

On the contrary, if nearly all cuts are selected then the time spent to select the cuts may not be compensated by the (small)
reduction in CPU time for solving the problems in the backward and forward passes.
This is the case of {\tt{CuSMuDA CS 1}} (see Figure \ref{fig:f_2}) where for all stages the mean proportion of cuts selected
is above 55\% and all cuts are selected at the last stage (as expected, recall Example \ref{excs1}) 
but also for some instances at other stages (stages $3$ and $4$ for the instance $(T,n,M)=(5,4,20)$,
stage $4$ for the instances $(T,n,M)=(5,5,20), (5,6,20), (8,5,20)$,
stages $6$ and $7$ for the instance $(T,n,M)=(8,4,10)$).

It is also interesting to note that in general, fewer cuts are selected for the first stages, especially for {\tt{CuSMuDA CS 1}}.
This makes sense since at the last stage, the cuts for functions $\mathfrak{Q}_T(\cdot, \xi_{T j})$ are exact (and as we recalled, all of them
are selected with  {\tt{CuSMuDA CS 1}}) but not necessarily for stages $t<T$ because for these stages approximate recourse functions
are used and the approximation errors {\em propagate} backward. Therefore it is expected that at the early stages old cuts (computed at the first
iterations using crude approximations of the recourse functions) will probably not be selected.

\begin{figure}
\centering
\begin{tabular}{cc}
\includegraphics[scale=0.6]{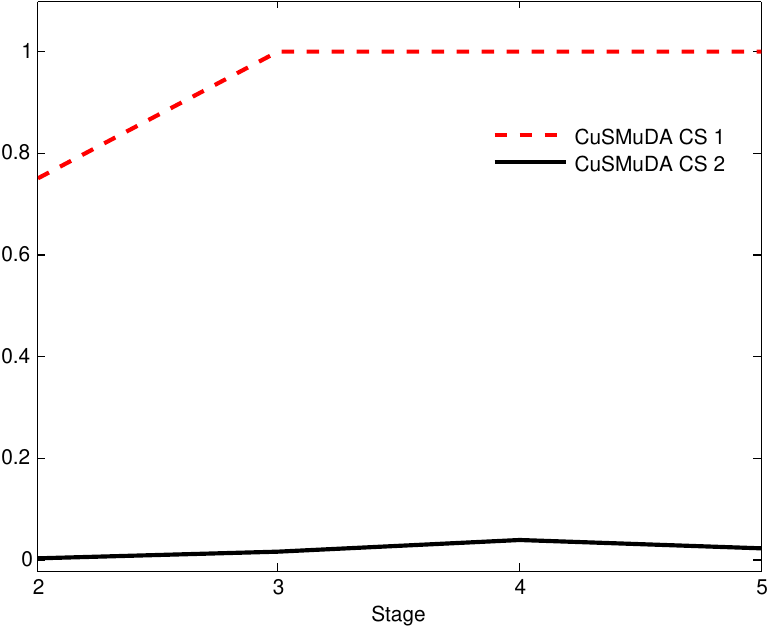}
&
\includegraphics[scale=0.6]{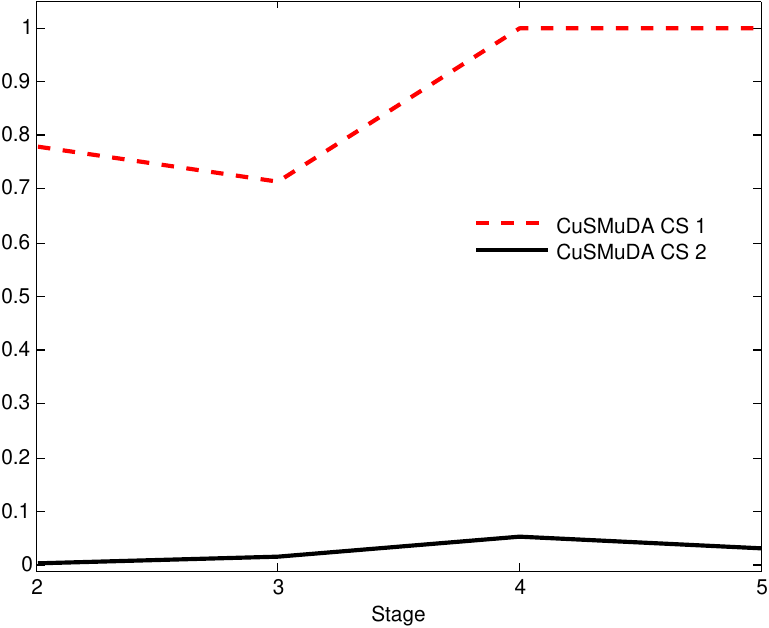}
\\
{$T=5, n=4, M=20$} & {$T=5, n=5, M=20$} \\
\includegraphics[scale=0.6]{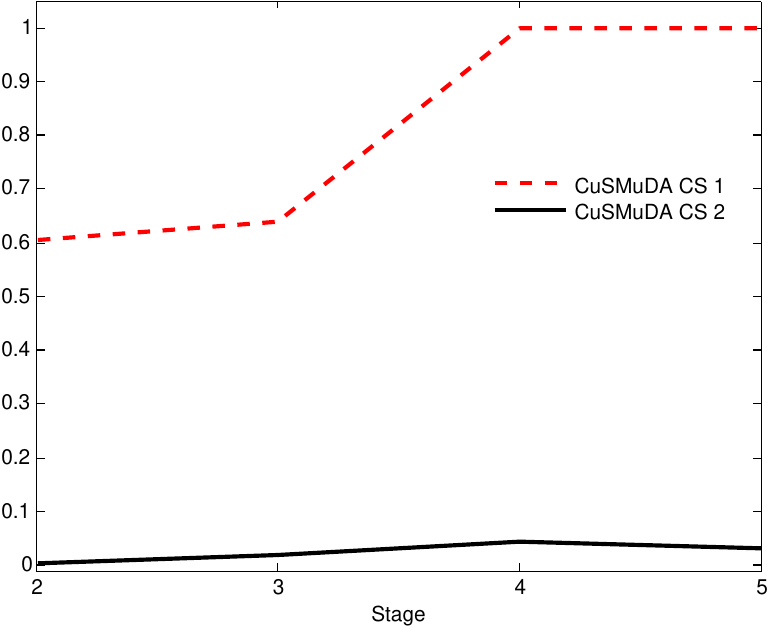}
&
\includegraphics[scale=0.6]{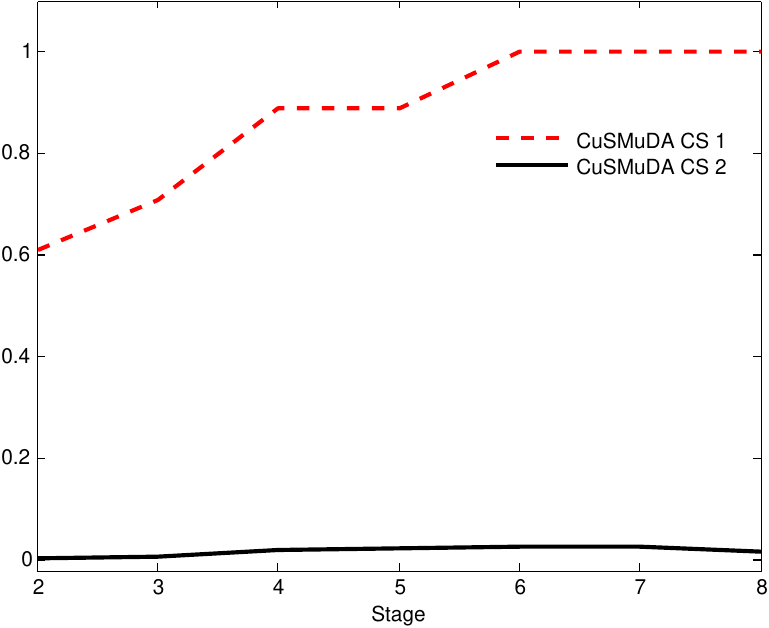}
\\
{$T=5, n=6, M=20$} & {$T=8, n=4, M=10$} \\
\includegraphics[scale=0.6]{Total_Nb_Cuts_T5_n5_M20.pdf}
&
\includegraphics[scale=0.6]{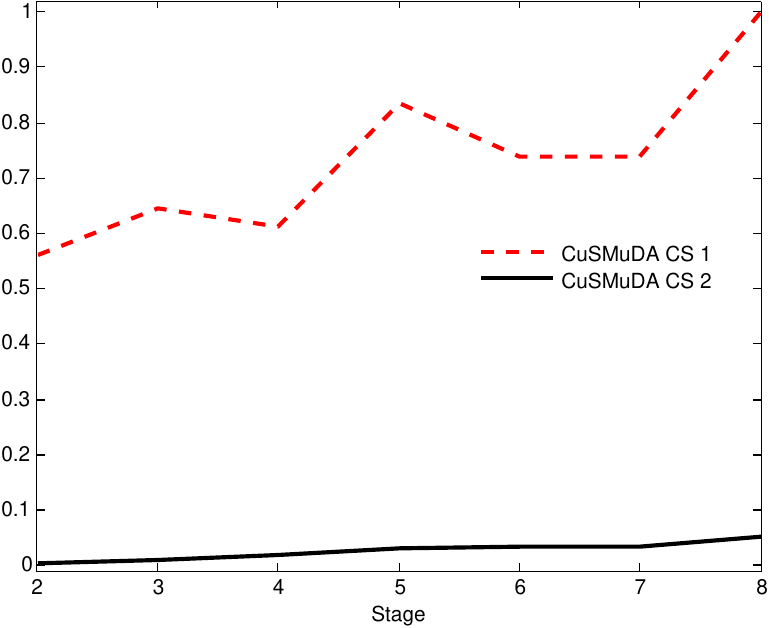}
\\
{$T=8, n=5, M=20$} & {$T=8, n=6, M=10$} \\
\end{tabular}
\caption{ \label{fig:f_2} Mean proportion of cuts (over the iterations of the algorithm) selected for stages $t=2,\ldots,T$ by 
{\tt{CuSMuDA CS 1}} and {\tt{CuSMuDA CS 2}}.}
\end{figure}

Similar conclusions can be drawn from  Figures \ref{fig:f_3}-\ref{fig:f_4} which
depict the evolution of the mean proportions of cuts (over the iterations of the algorithm) selected for 
$\mathfrak{Q}_t(\cdot, \xi_{t j})$ as a function of $j=1,\ldots,M$. 
Additionally, we see that for each stage $t$, the proportions of cuts selected 
for functions $\mathfrak{Q}_t(\cdot, \xi_{t j}), j=1,\ldots,M,$ are quite similar.

\begin{figure}
\centering
\begin{tabular}{cc}
\includegraphics[scale=0.6]{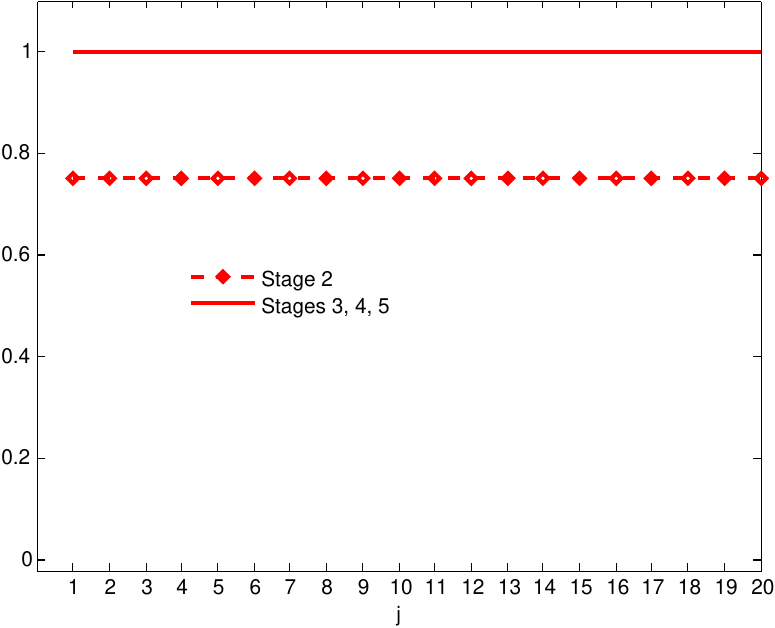}
&
\includegraphics[scale=0.64]{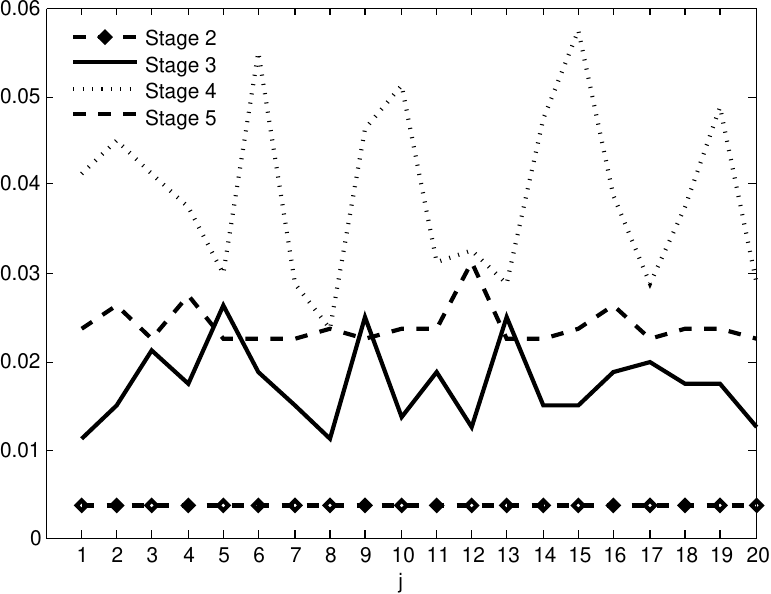}
\\
{$T=5, n=4, M=20$, {\tt{CuSMuDA CS 1}}} & {$T=5, n=4, M=20$, {\tt{CuSMuDA CS 2}}}
\\
\includegraphics[scale=0.6]{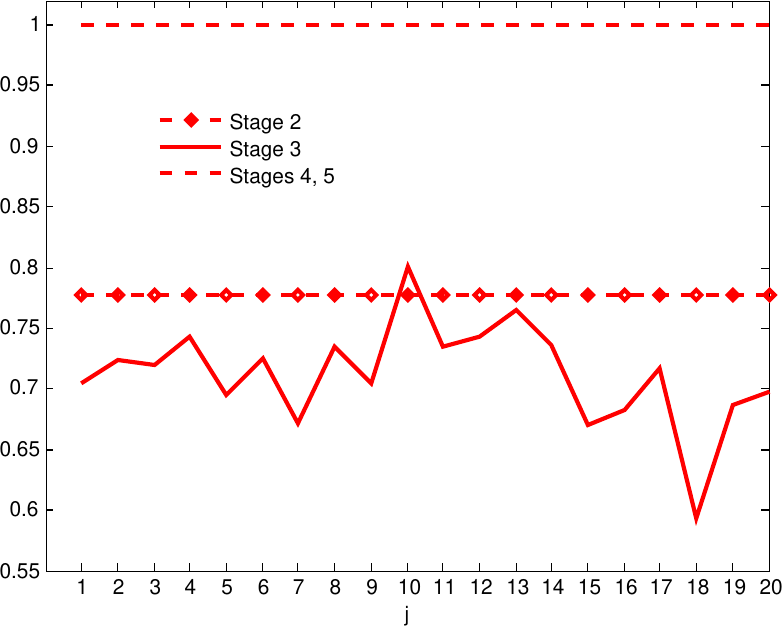}
&
\includegraphics[scale=0.6]{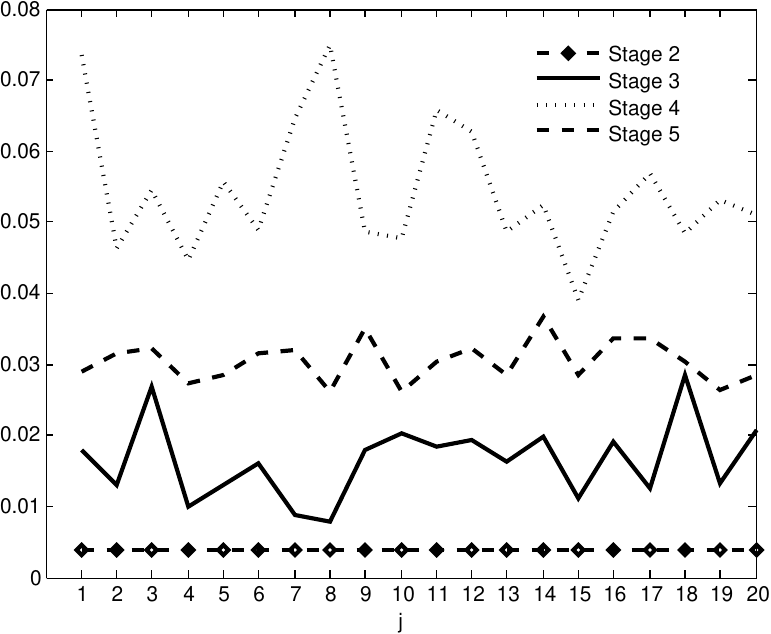}
\\
{$T=5, n=5, M=20$, {\tt{CuSMuDA CS 1}}} & {$T=5, n=5, M=20$, {\tt{CuSMuDA CS 2}}}
\\
\includegraphics[scale=0.6]{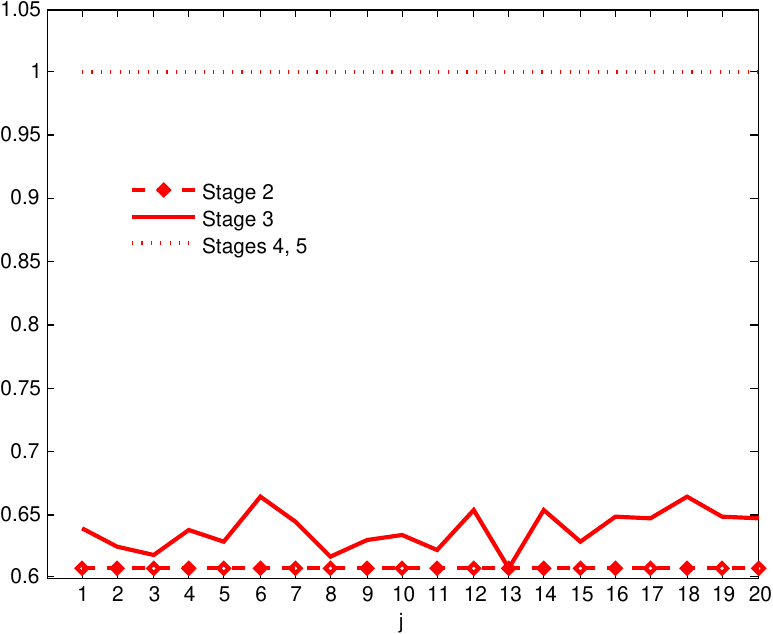}
&
\includegraphics[scale=0.6]{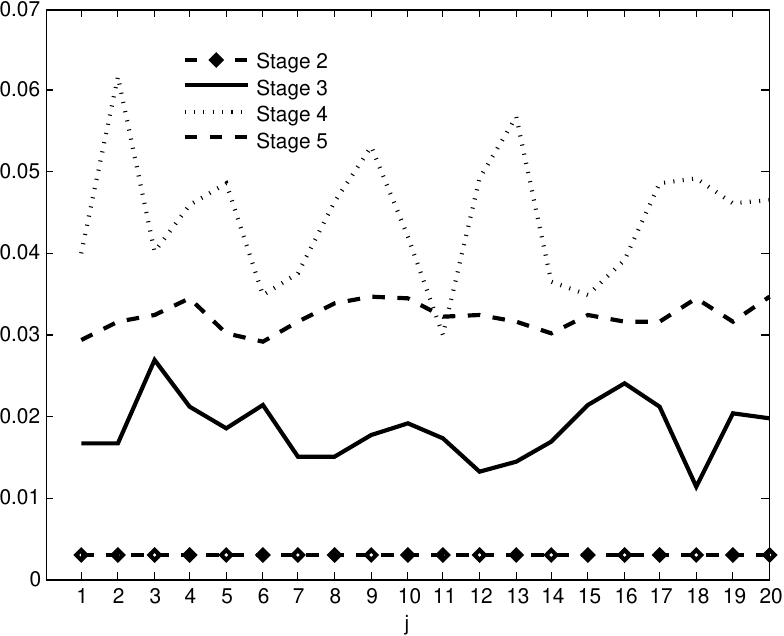}\\
{$T=5, n=6, M=20$, {\tt{CuSMuDA CS 1}}} & {$T=5, n=6, M=20$, {\tt{CuSMuDA CS 2}}}
\end{tabular}
\caption{ \label{fig:f_3}  Representation for stages $t=2,\ldots,T$, of the 
mean proportion of cuts (over the iterations of the algorithm) selected for 
$\mathfrak{Q}_t(\cdot, \xi_{t j})$ as a function of $j=1,\ldots,M$.
Left plots: {\tt{CuSMuDA CS 1}}, right plots: {\tt{CuSMuDA CS 2}}.}
\end{figure}

\begin{figure}
\centering
\begin{tabular}{cc}
\includegraphics[scale=0.6]{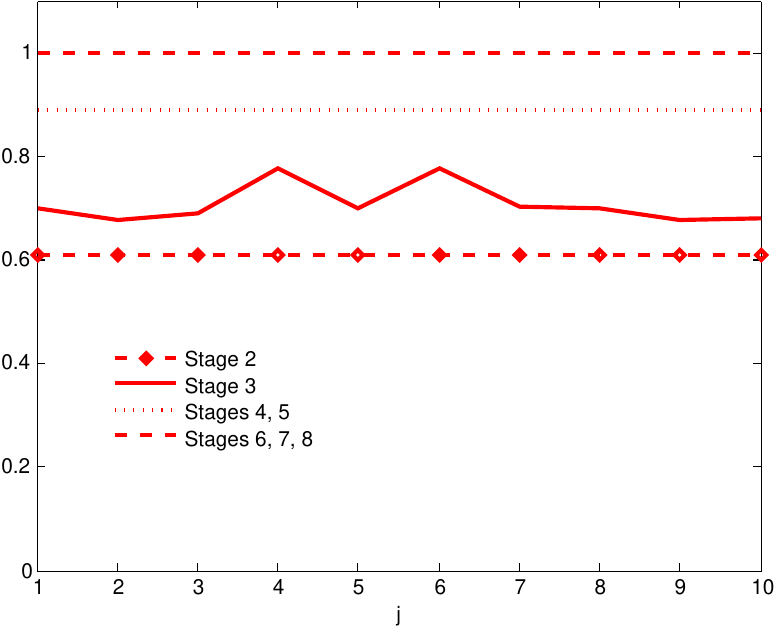}
&
\includegraphics[scale=0.6]{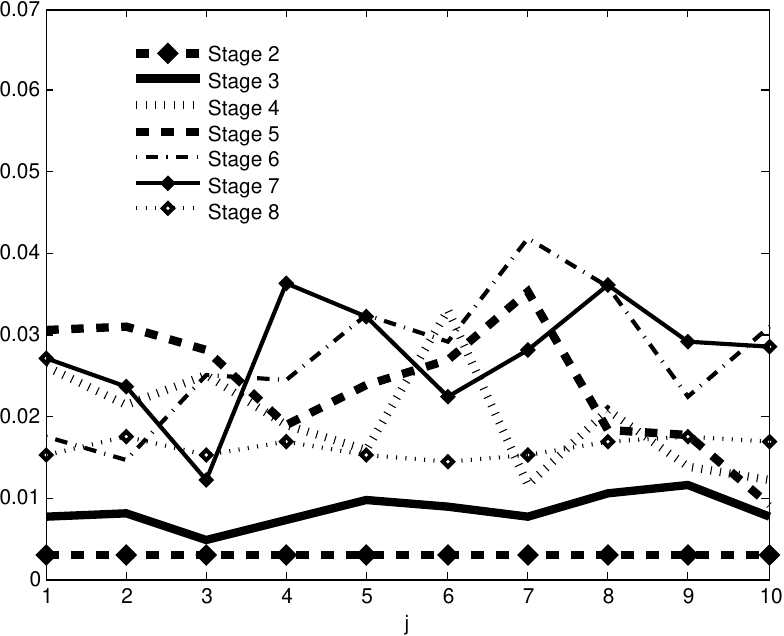}\\
{$T=8, n=4, M=10$, {\tt{CuSMuDA CS 1}}} & {$T=8, n=4, M=10$, {\tt{CuSMuDA CS 2}}}\\
\includegraphics[scale=0.6]{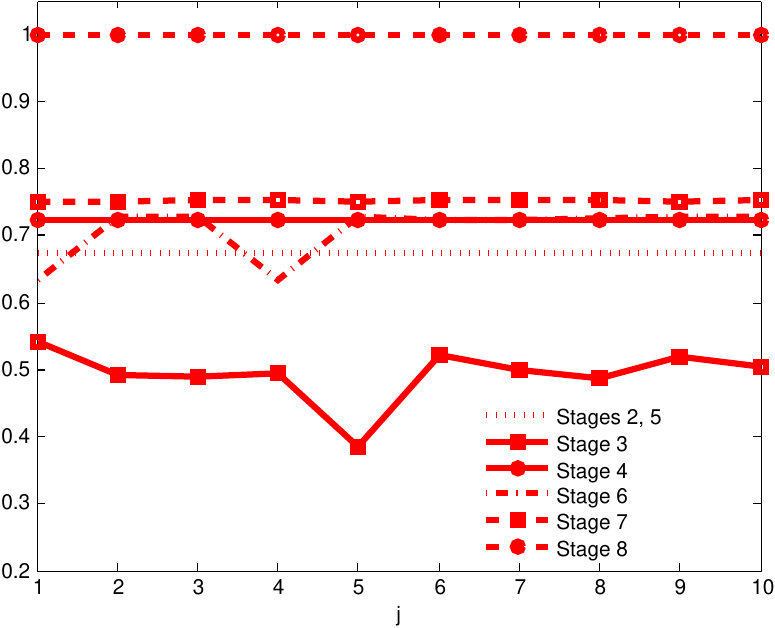}
&
\includegraphics[scale=0.6]{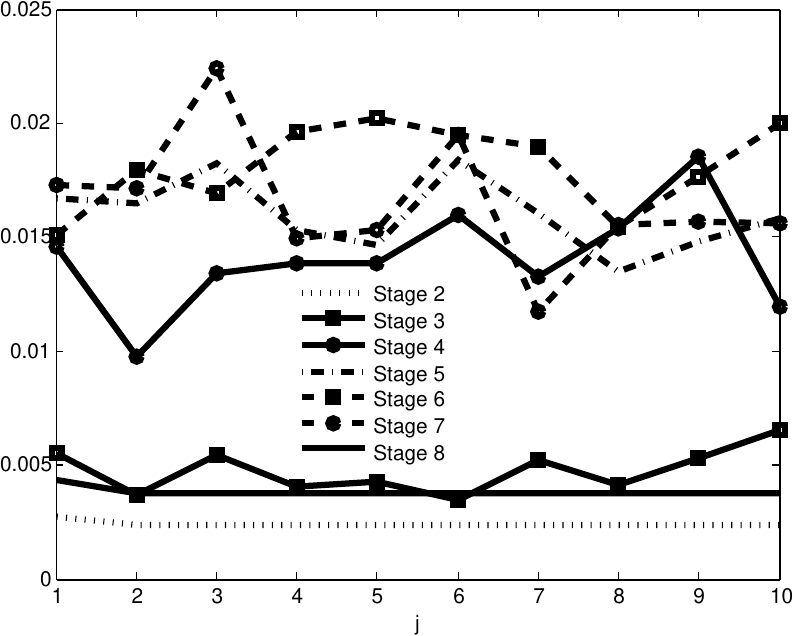}\\
{$T=8, n=5, M=10$, {\tt{CuSMuDA CS 1}}} & {$T=8, n=5, M=10$, {\tt{CuSMuDA CS 2}}}\\
\includegraphics[scale=0.6]{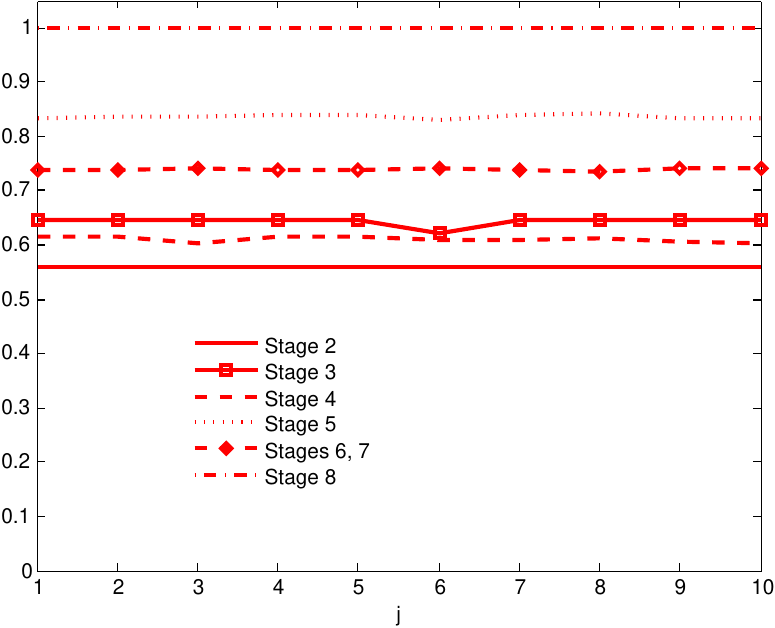}
&
\includegraphics[scale=0.6]{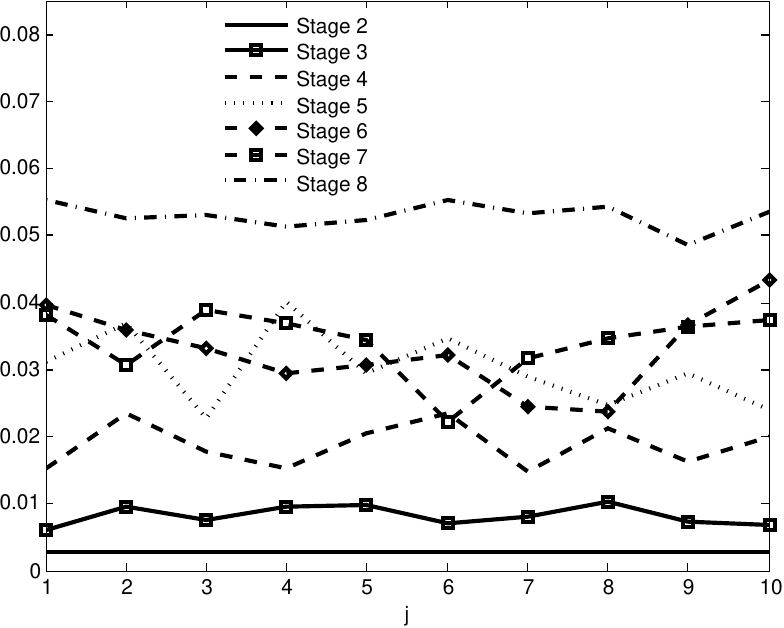}\\
{$T=8, n=6, M=10$, {\tt{CuSMuDA CS 1}}} & {$T=8, n=6, M=10$, {\tt{CuSMuDA CS 2}}}
\end{tabular}
\caption{ \label{fig:f_4} Representation for stages $t=2,\ldots,T$, of the 
mean proportion of cuts (over the iterations of the algorithm) selected for 
$\mathfrak{Q}_t(\cdot, \xi_{t j})$ as a function of $j=1,\ldots,M$.
Left plots: {\tt{CuSMuDA CS 1}}, right plots: {\tt{CuSMuDA CS 2}}.}
\end{figure}

One last comment is now in order. We have already observed that on all
experiments, the Multicut Level 1 and Territory Algorithm cut selection strategies, i.e., {\tt{CuSMuDA CS 1}},
correctly select all cuts at the final stage. However, a crude implementation of the
Multicut Level 1 pseudo-code given in Figure \ref{figurecut1} resulted in 
the elimination of cuts at the final stage. This comes from the fact that approximate solutions of the
optimization subproblems are computed. Therefore two optimization problems could  compute the same cuts but the solver
may return two different (very close) approximate solutions. Similarly, a cut may be in theory the highest at some point
(for instance cut $\mathcal{C}_{T j}^k$ is, in theory, the highest at $x_{T-1}^k$) but numerically the value of another
cut at this point may be deemed slightly higher, because of numerical errors.
The remedy is to introduce a small error term $\varepsilon_0$ ($\varepsilon_0=10^{-6}$ in our experiments)
such that the values $V_1$ and $V_2$ of two cuts $\mathcal{C}_1$ and $\mathcal{C}_2$  at a trial point are considered equal if
$|V_2 -V_1| \leq \varepsilon_0 \max(1, |V_1|)$ while $\mathcal{C}_1$ is above $\mathcal{C}_2$ at this trial point if
$V_1 \geq  V_2 + \varepsilon_0 \max(1, |V_1|)$.
Therefore the pseudo-codes of Multicut Level 1 and MLM Level 1 
given in Figure \ref{figurecut1} need to be modified. The corresponding pseudo-codes taking into account the approximation
errors are given in Figure \ref{figurecut2} in the Appendix.

\subsection{Inventory management}

\subsubsection{Model} \label{sec:inventory}

We consider the $T$-stage inventory management problem given in (Shapiro et al. 2009).
For each stage $t=1,\ldots, T$, on the basis of the inventory
level $y_t$ at the beginning of period $t$, we have to decide on the quantity $x_t - y_t$ of a product
to buy so that the inventory level becomes $x_t$. 
Given demand $\xi_t$ for that product 
for stage $t$, the  inventory level is $y_{t+1}=x_t-\xi_t$ at the beginning of stage $t+1$.
The inventory level can become negative, in which case a backorder cost is paid.
If one is interested in minimizing the average cost over the optimization period,
we need to solve the following dynamic programming equations: for $t=1,\ldots,T$, defining
$\mathcal{Q}_t( y_t ) = \mathbb{E}_{\xi_t}[\mathfrak{Q}_t( y_t, \xi_t)]$, the stage $t$ problem is
$$
\mathfrak{Q}_t(y_t, \xi_t)=
\left\{ 
\begin{array}{l}
\inf c_t( x_t - y_t ) + b_t (\xi_t - x_t)_{+} + h_t (x_t - \xi_t)_{+} + \mathcal{Q}_{t+1}(y_t)\\
y_{t+1} = x_t -\xi_t, x_t \geq y_t,
\end{array}
\right.
$$
where $c_t$ is the unit buying cost, $h_t$ is the holding cost, and $b_t$ the backorder cost.
The optimal mean cost is $\mathcal{Q}_1( y_1 )$ where $y_1$ is the initial stock.
In what follows we present the results of numerical simulations obtained solving this problem with
MuDA with sampling (denoted again by {\tt{MuDA}}), CuSMuDA with Mulicut Level 1 cut selection (denoted again by {\tt{CuSMuDA CS 1}}), and CuSMuDA with 
MLM Level 1 cut selection  (denoted again by {\tt{CuSMuDA CS 2}}).

\subsubsection{Numerical results}

We consider six values for the number of stages $T$ ($T \in \{5, 10, 15, 20, 25, 30\}$) and for fixed $T$, the following values of the
problem and algorithm parameters are taken:
\begin{itemize} 
\item $c_t=1.5+\cos(\frac{\pi t}{6}), b_t=2.8, h_t=0.2, M_t=M$ for all $t=1,\ldots,T$,
\item $y_1=10$, $p_{t i}=\frac{1}{M}=\frac{1}{20}$ for all $t,i$,
\item $\xi_1= {\overline{\xi_1}}$ and $(\xi_{t 1},\xi_{t 2},\ldots,\xi_{t M})$ 
corresponds to a sample from the distribution of ${\overline{\xi_t}}(1+ 0.1\varepsilon_t)$ for i.i.d $\varepsilon_t \sim \mathcal{N}(0, 1)$
for $t=2,\ldots,T$, where ${\overline{\xi_t}}=5+0.5t$,
\item for the stopping criterion, in \eqref{formulazsup} $\alpha=0.025$
and $N=200$, and $\varepsilon=0.05$ in \eqref{stoppingcriterion}.
\end{itemize}

\par {\textbf{Checking the implementations.}}
As for the portfolio problem, for each value of $T \in \{10, 15, 20, 25, 30\}$, we plot in Figure \ref{fig:f_5}
the evolution of the upper and lower bounds along the iterations of the three algorithms.
These bounds are very close for the three algorithms at the final iteration.
For these runs, the bounds are in fact very close for all iterations, not only the last one.
For $T=5$ only one iteration was necessary for all three algorithms, yielding the bounds 
$[25.25, 25.61], [25.25, 25.56]$, and $[25.25, 25.62]$, for respectively 
{\tt{MuDA}}, {\tt{CuSMuDA CS 1}}, and {\tt{CuSMuDA CS 2}} at the end of this iteration. \\

\begin{figure}
\centering
\begin{tabular}{cc}
\includegraphics[scale=0.6]{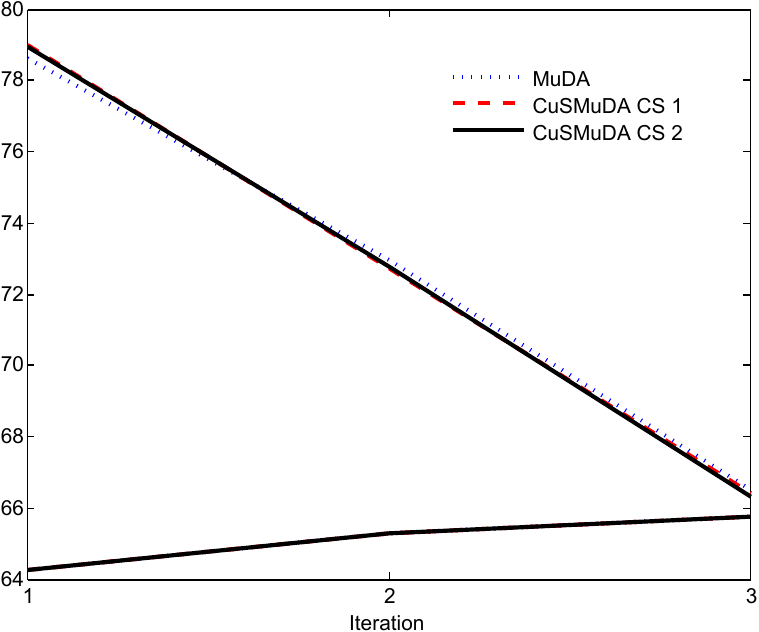}
&
\includegraphics[scale=0.6]{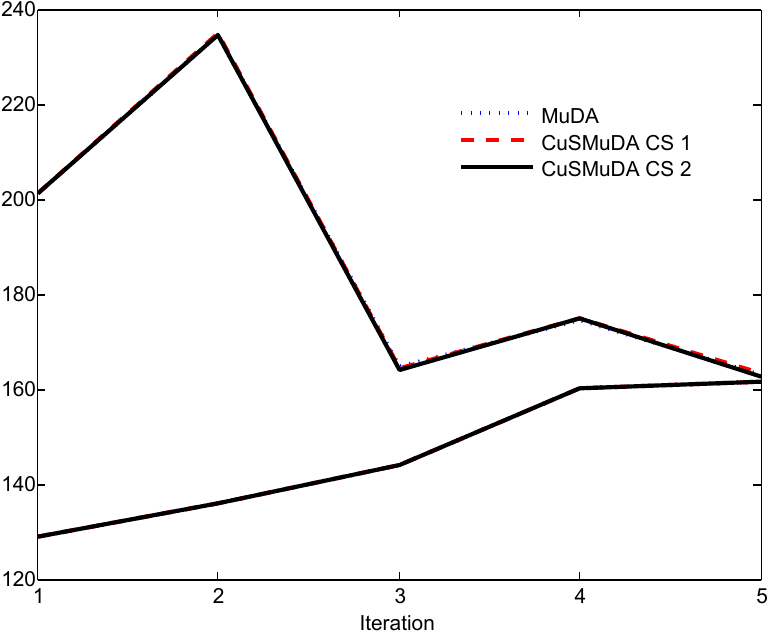}
\\
{$T=10, M=20$} & {$T=15, M=20$}\\
\includegraphics[scale=0.6]{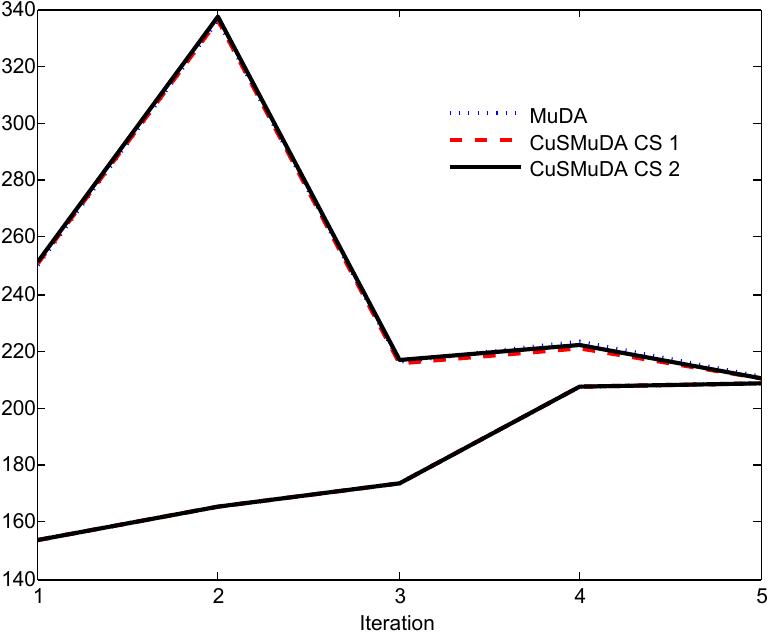}
&
\includegraphics[scale=0.6]{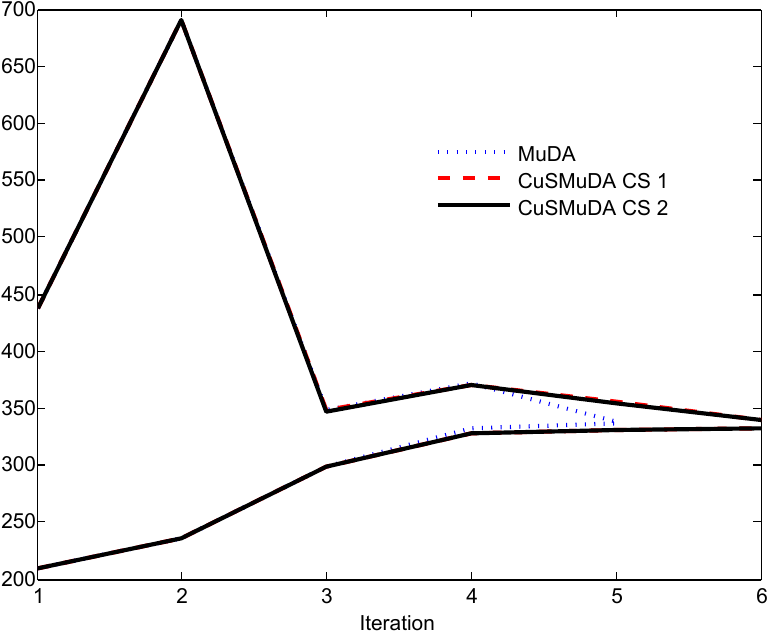}
\\
{$T=20, M=20$} &  {$T=25, M=20$}\\
\end{tabular}
\begin{tabular}{c}
\includegraphics[scale=0.6]{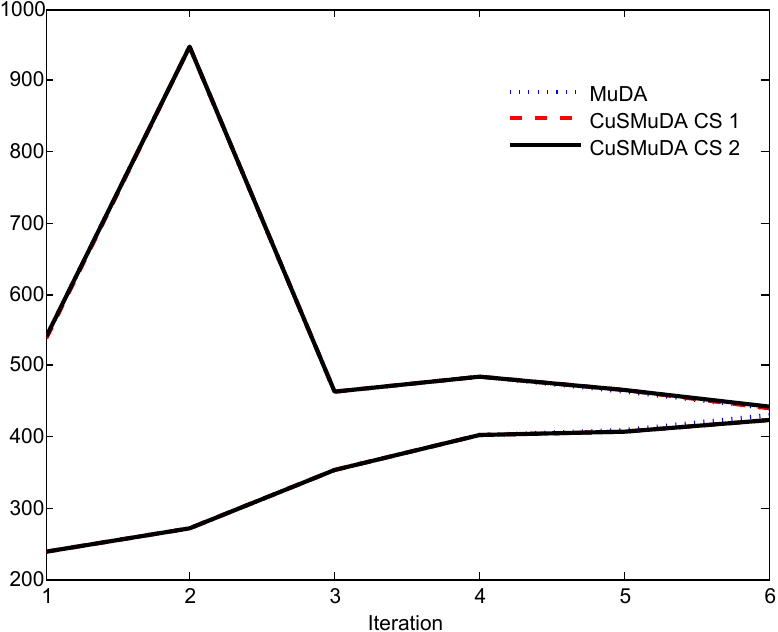}
\\
{$T=30, M=20$} \\
\end{tabular}
\caption{ \label{fig:f_5} Evolution of the upper bounds $z_{\sup}^i$ and lower bounds $z_{\inf}^i$
along the iterations of the algorithms.}
\end{figure}

\par {\textbf{Computational time.}} 
Table \ref{tablerunningtime2} shows the CPU time needed to solve our six instances
of inventory management problems. As for the portfolio problem and for the same reasons, 
{\tt{CuSMuDA CS 2}} is much quicker than both {\tt{MuDA}} (between 6.4 and 15.7 times quicker)
and {\tt{CuSMuDA CS 1}} (between 10.2 and 25.6  times quicker).\\

\begin{table}
\centering
\begin{tabular}{|c||c|c|c|}
\hline
& {\tt{MuDA}} & {\tt{CuSMuDA CS 1}}  & {\tt{CuSMuDA CS 2}} \\
\hline
$T=5, M=20$ &1.57  &  2.90  &  0.22   \\
\hline
$T=10, M=20$ & 26.33 &   43.07   & 1.68  \\
\hline
$T=15, M=20$ & 106.52  &  178.08    & 8.07  \\
\hline
$T=20, M=20$ &  158.10  &  245.24    & 19.43 \\
\hline
$T=25, M=20$ &189.27   &  405.58   & 22.64 \\
\hline
$T=30, M=20$ & 363.86   & 575.53    & 56.62 \\
\hline
\end{tabular}
\caption{Computational time (in minutes) for solving instances of the inventory  problem of Section \ref{sec:inventory} with {\tt{MuDA}}, {\tt{CuSMuDA CS 1}}, and {\tt{CuSMuDA CS 2}}.}
\label{tablerunningtime2}
\end{table}

\par {\textbf{Proportion of cuts selected.}} 
The proportions of cuts selected exhibit similar patterns to those observed
on the portfolio problems. In Figure \ref{fig:f_6} which reports for all instances 
the mean proportion of cuts selected for all stages, we observe again that for all stages on average
at least 40\% of cuts are selected for {\tt{CuSMuDA CS 1}} and at most 20\% of cuts are selected for 
{\tt{CuSMuDA CS 2}}. Figures \ref{fig:f_7} and \ref{fig:f_8} show
the evolution of the mean proportions of cuts (over the iterations of the algorithm) selected for 
$\mathfrak{Q}_t(\cdot, \xi_{t j})$ as a function of $j=1,\ldots,M$. 
The conclusions are similar to those of Section \ref{portfolio}:
\begin{enumerate}
\item for most stages, the proportion of cuts selected with {\tt{CuSMuDA CS 1}} is very small, below 2\% (see the
right plots of Figures \ref{fig:f_7} and \ref{fig:f_8}).
\item With {\tt{CuSMuDA CS 1}} for the last stage but also for other stages on some instances (stages $10-14$ for $T=15$, stages $10-12$ for $T=20$,
stages $25-29$ for $T=30$) all cuts
are selected at all iterations and at several other stages a large proportion (above 80\%) of cuts is selected.
\item For each stage $t$, the proportions of cuts selected 
for functions $\mathfrak{Q}_t(\cdot, \xi_{t j}), j=1,\ldots,M,$ are quite similar.
\end{enumerate}
For $T=5$, where one iteration is enough to satisfy the stopping criterion for all algorithms,
{\tt{CuSMuDA CS 1}} selects all cuts for all stages while {\tt{CuSMuDA CS 2}}
selects for all stages in the backward pass only  0.5\%
of the cuts, i.e., (0.5/100)$\small{\times}200\small{\times}20=20$ cuts per stage. The observations above explain why {\tt{CuSMuDA CS 2}} is much quicker.

\begin{figure}
\centering
\begin{tabular}{cc}
\includegraphics[scale=0.6]{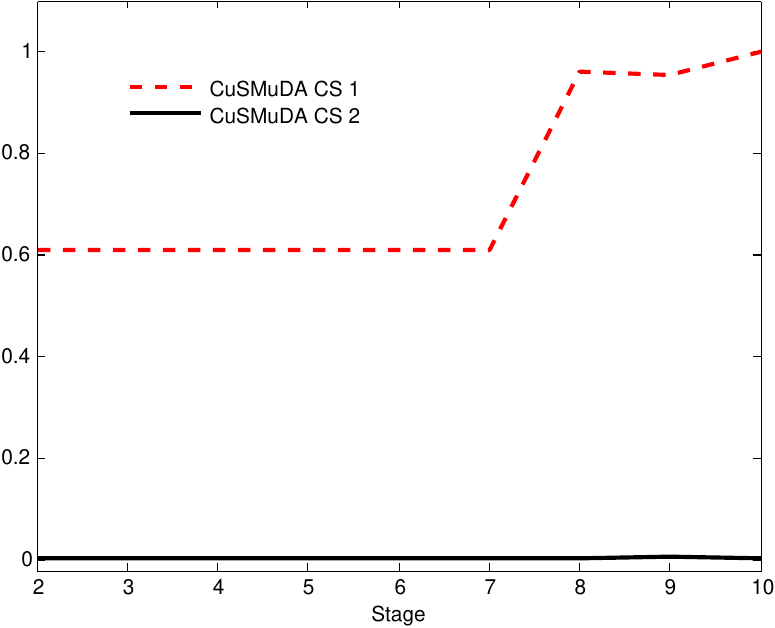}
&
\includegraphics[scale=0.6]{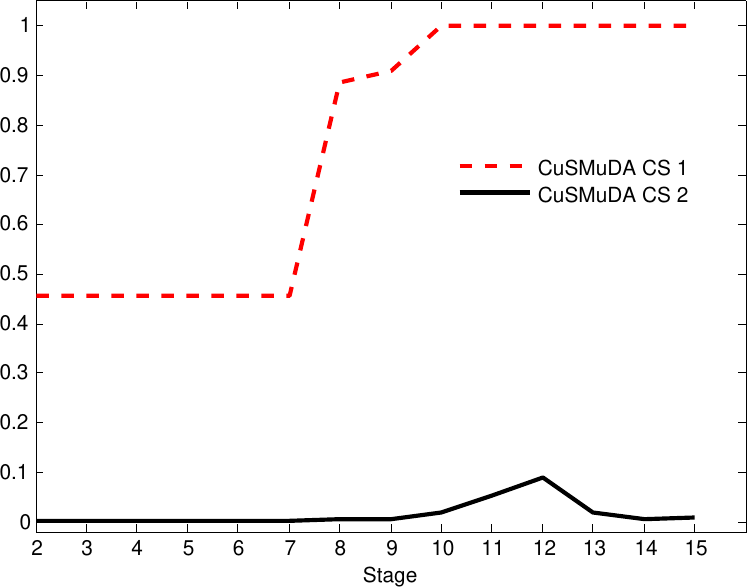}
\\
{$T=10, M=20$} & {$T=15, M=20$} \\
\includegraphics[scale=0.6]{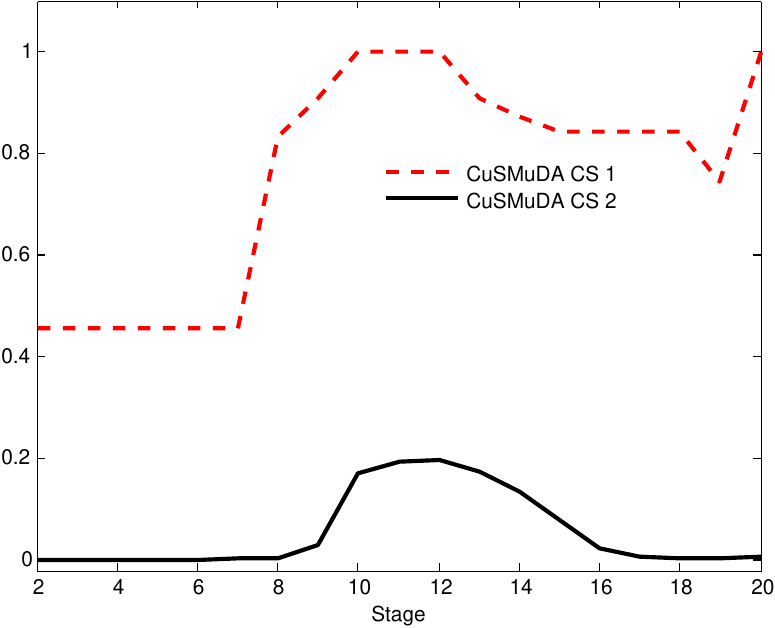}
&
\includegraphics[scale=0.6]{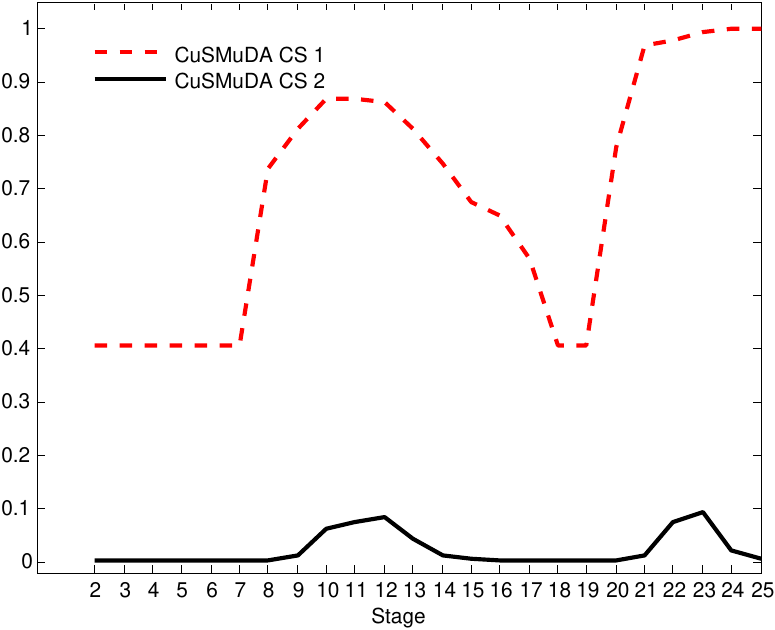}
\\
{$T=20, M=20$} & {$T=25, M=20$} 
\end{tabular}
\begin{tabular}{c}
\includegraphics[scale=0.6]{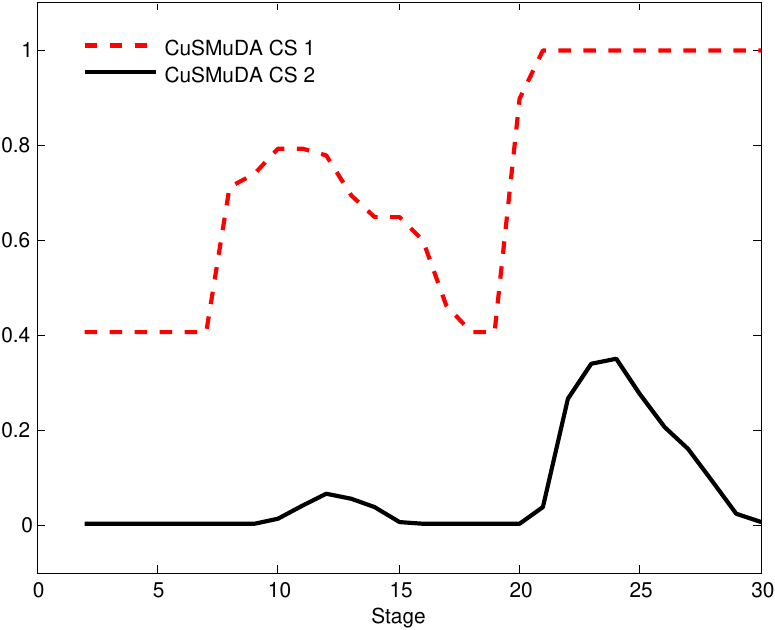}
\\
{$T=30, M=20$} 
\end{tabular}
\caption{ \label{fig:f_6} Mean proportion of cuts (over the iterations of the algorithm) selected for stages $t=2,\ldots,T$ for 
{\tt{CuSMuDA CS 1}} and {\tt{CuSMuDA CS 2}}.}
\end{figure}

\begin{figure}
\centering
\begin{tabular}{cc}
\includegraphics[scale=0.6]{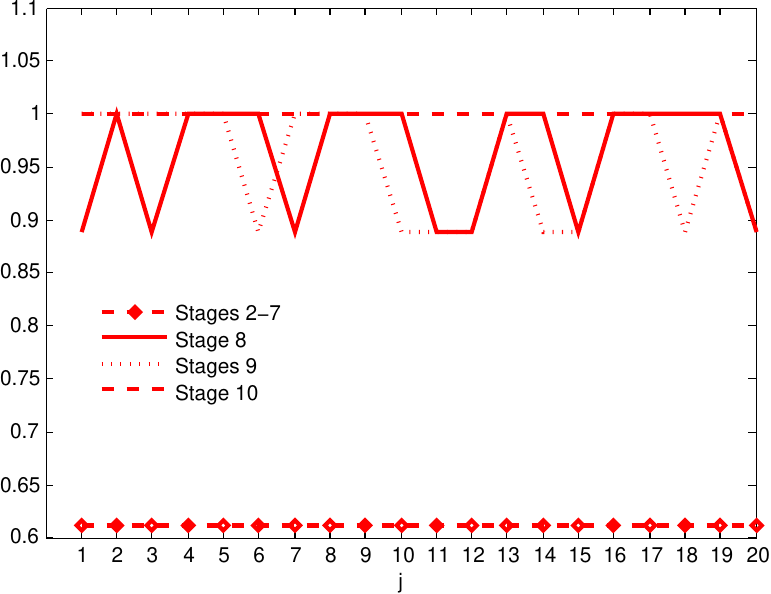}
&
\includegraphics[scale=0.57]{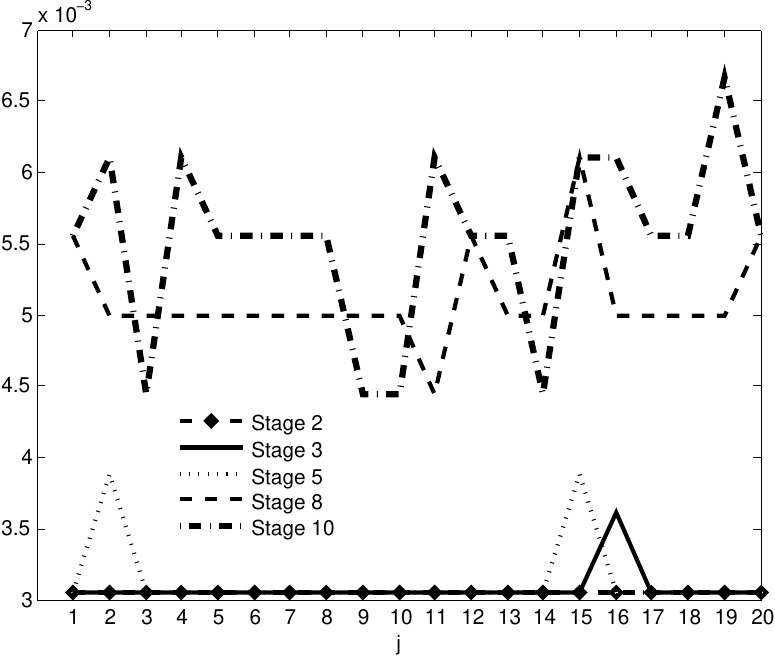}
\\
{$T=10, M=20$, {\tt{CuSMuDA CS 1}}} & {$T=10, M=20$, {\tt{CuSMuDA CS 2}}}
\\
\includegraphics[scale=0.6]{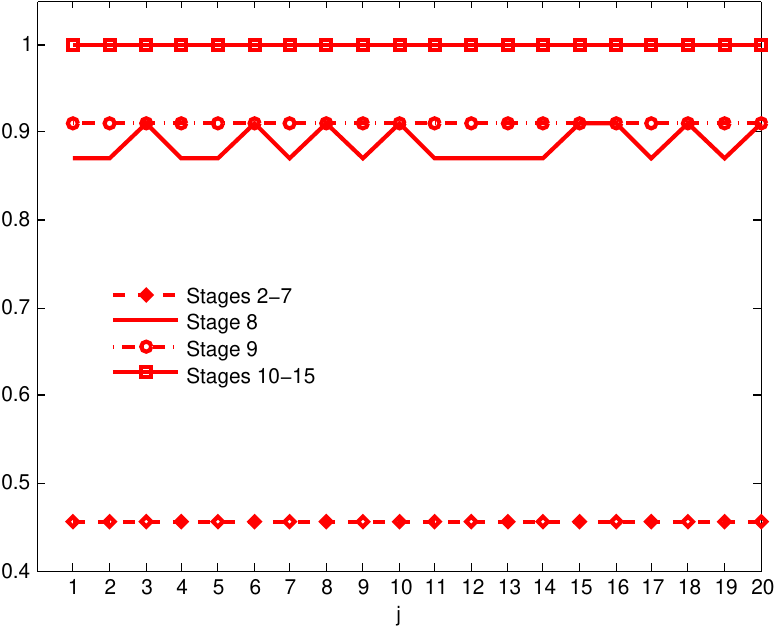}
&
\includegraphics[scale=0.6]{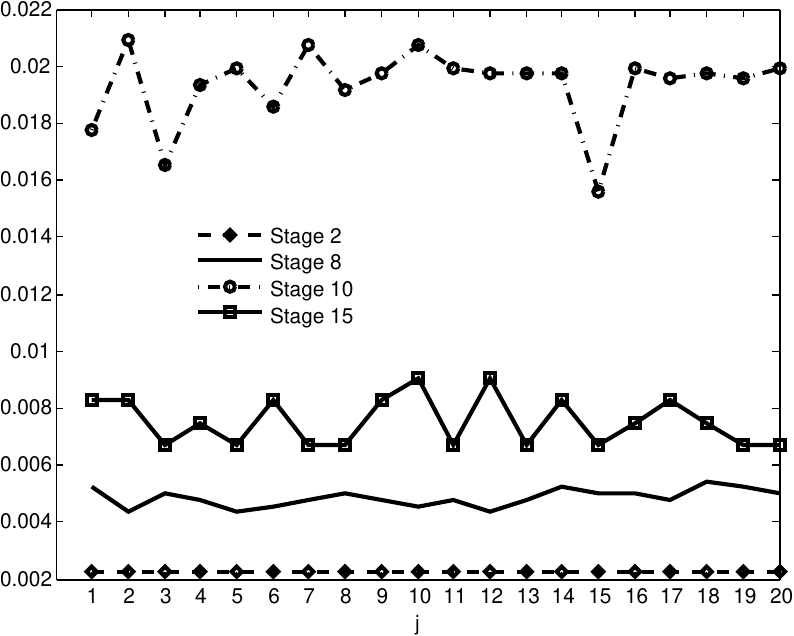}\\
{$T=15, M=20$, {\tt{CuSMuDA CS 1}}} & {$T=15, M=20$, {\tt{CuSMuDA CS 2}}}\\
\includegraphics[scale=0.6]{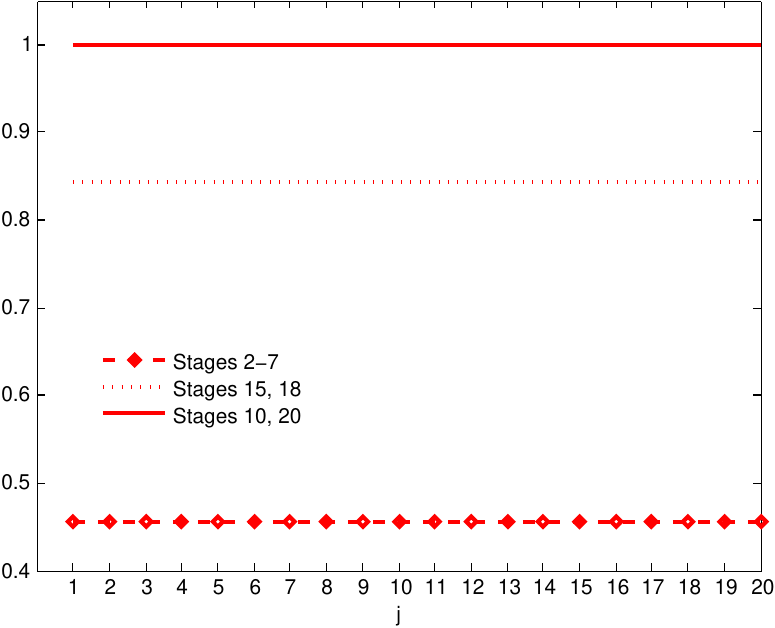}
&
\includegraphics[scale=0.6]{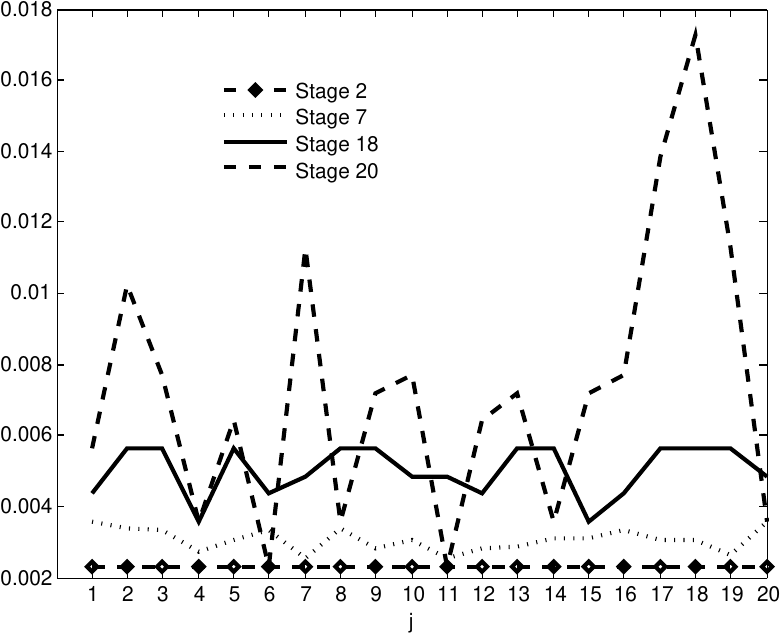}
\\
{$T=20, M=20$, {\tt{CuSMuDA CS 1}}} & {$T=20, M=20$, {\tt{CuSMuDA CS 2}}}
\end{tabular}
\caption{ \label{fig:f_7}  {\small{Representation for some stages $t \in \{2,\ldots,T\}$, of the 
mean proportion of cuts (over the iterations of the algorithm) selected for 
$\mathfrak{Q}_t(\cdot, \xi_{t j})$ as a function of $j=1,\ldots,M$.
Left plots: {\tt{CuSMuDA CS 1}}, right plots: {\tt{CuSMuDA CS 2}}}}.}
\end{figure}

\begin{figure}
\centering
\begin{tabular}{cc}
\includegraphics[scale=0.6]{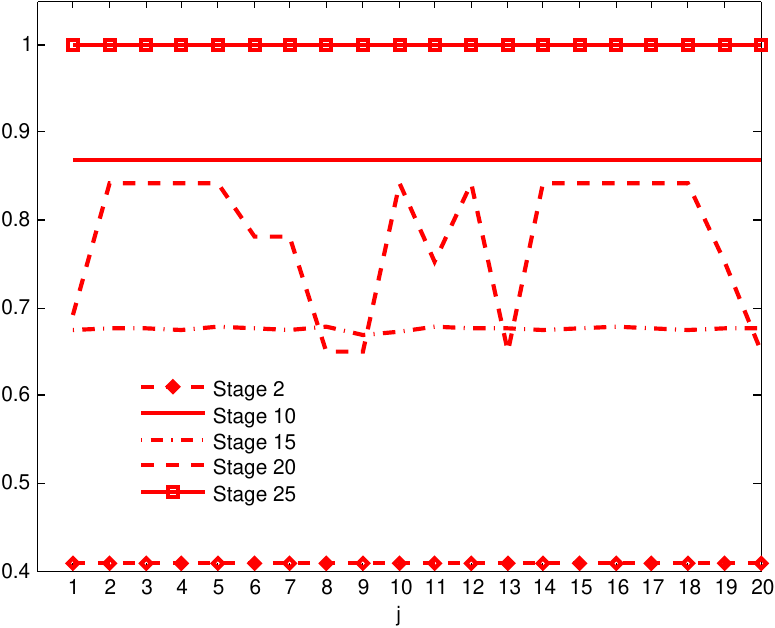}
&
\includegraphics[scale=0.6]{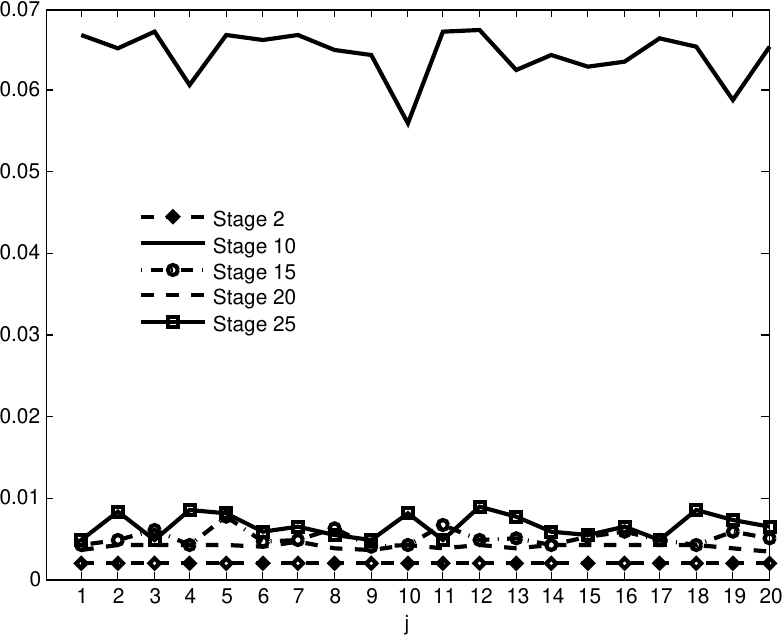}
\\
{$T=25, M=20$, {\tt{CuSMuDA CS 1}}} & {$T=25, M=20$, {\tt{CuSMuDA CS 2}}}
\\
\includegraphics[scale=0.6]{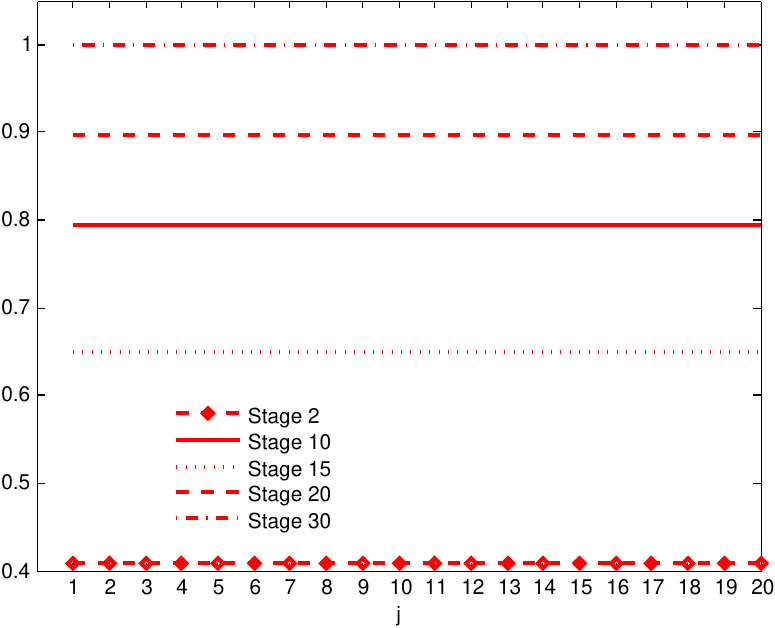}
&
\includegraphics[scale=0.6]{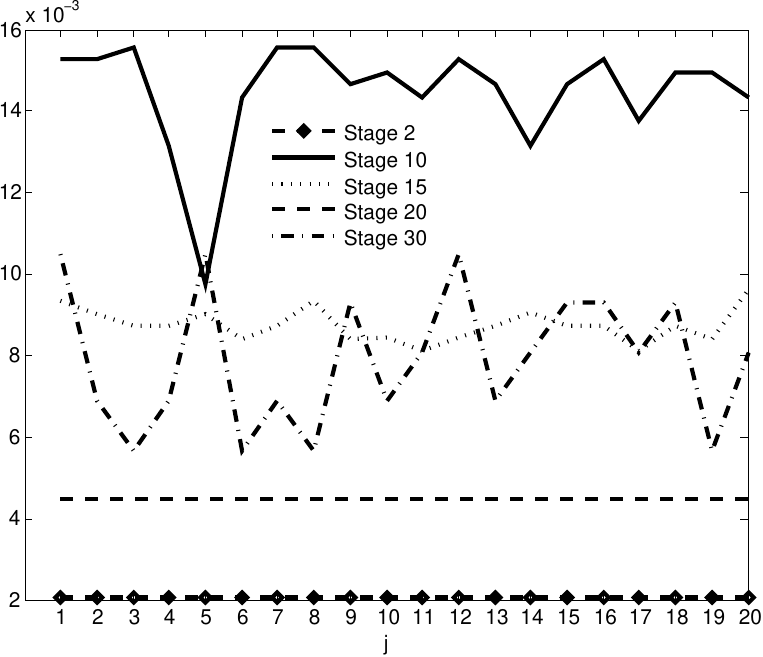}\\
{$T=30, M=20$, {\tt{CuSMuDA CS 1}}} & {$T=30, M=20$, {\tt{CuSMuDA CS 2}}}
\end{tabular}
\caption{ \label{fig:f_8}  {\small{Representation for some stages $t \in \{2,\ldots,T\}$, of the 
mean proportion of cuts (over the iterations of the algorithm) selected for 
$\mathfrak{Q}_t(\cdot, \xi_{t j})$ as a function of $j=1,\ldots,M$.
Left plots: {\tt{CuSMuDA CS 1}}, right plots: {\tt{CuSMuDA CS 2}}.}}}
\end{figure}

\section{Conclusion}

We proposed CuSMuDA, a  combination of a limited memory variant of the Level 1 cut selection strategy (MLM Level 1)
with Multicut Decomposition algorithms to solve multistage stochastic linear programs.
The convergence of the method in a finite number of iterations was proved for a class of cut selection
strategies that contains MLM Level 1. 
Compared to traditional Multicut Decomposition methods, CuSMuDA significantly reduced  
the computational time on all six instances of a portfolio problem
and six instances of an inventory problem. 

It would be interesting to test CuSMuDA and the limited memory variant of Level 1 (both for MuDA and SDDP) on other types of stochastic programs and to extend the analysis
to nonlinear stochastic programs.

\section*{Acknowledgments} The second author's research was 
partially supported by an FGV grant, CNPq grants 307287/2013-0 and 401371/2014-0, and FAPERJ grant E-26/201.599/2014.
The authors wish to thank Vincent Lecl\`ere for helpful discussions.

\nocite{*}

\addcontentsline{toc}{section}{References}
\bibliographystyle{plain}
\bibliography{Bibliography}

\begin{thebibliography}{10}

\bibitem{mosek}
E.~D. Andersen and K.D. Andersen.
\newblock {\em The MOSEK optimization toolbox for MATLAB manual. Version 7.0},
  2013.
\newblock \url{https://www.mosek.com/}.

\bibitem{powellasamov}
T.~Asamov and W.~Powell.
\newblock Regularized decomposition of high-dimensional multistage stochastic
  programs with markov uncertainty.
\newblock {\em Available at: {{\tt https://arxiv.org/abs/1505.02227}}}, 2015.

\bibitem{bentalmargnem00}
A.~Ben-Tal, T.~Margalit, and A.~Nemirovski.
\newblock Robust modeling of multi-stage portfolio problems.
\newblock {\em in: H. Frenk, K. Roos, T. Terlaky, S. Zhang, Eds., High
  Performance Optimization, Kluwer Academic Publishers}, pages 303--328, 2000.

\bibitem{BEHL}
M.~Best and J.~Hlouskova.
\newblock An algorithm for portfolio optimization with variable transaction
  costs, part 1: Theory.
\newblock {\em Journal of Optimization Theory and Applications}, 135:563--581,
  2007.

\bibitem{birge-louv-book}
J.~Birge and F.~Louveaux.
\newblock {\em {Introduction to Stochastic Programming}}.
\newblock Springer-Verlag, New York, 1997.

\bibitem{birgemulti}
J.R. Birge.
\newblock Decomposition and partitioning methods for multistage stochastic
  linear programs.
\newblock {\em Oper. Res.}, 33:989--1007, 1985.

\bibitem{birgedono}
J.R. Birge and C.~J. Donohue.
\newblock {The Abridged Nested Decomposition Method for Multistage Stochastic
  Linear Programs with Relatively Complete Recourse}.
\newblock {\em Algorithmic of Operations Research}, 1:20--30, 2001.

\bibitem{birge1996}
J.R. Birge, C.J. Donohue, D.F. Holmes, and O.G. Svintsitski.
\newblock A parallel implementation of the nested decomposition algorithm for
  multistage stochastic linear programs.
\newblock {\em Mathematical Programming}, 75:327--352, 1996.

\bibitem{birgelouv}
J.R. Birge and F.V. Louveaux.
\newblock A multicut algorithm for two-stage stochastic linear programs.
\newblock {\em European Journal of Operational Research}, 34:384--392, 1988.

\bibitem{chenpowell99}
Z.L. Chen and W.B. Powell.
\newblock {Convergent Cutting-Plane and Partial-Sampling Algorithm for
  Multistage Stochastic Linear Programs with Recourse}.
\newblock {\em {J. Optim. Theory Appl.}}, 102:497--524, 1999.

\bibitem{gassmanmp90}
H.~Gassmann.
\newblock Mslip: A computer code for the multistage stochastic linear
  programming problem.
\newblock {\em Math. Program.}, 47:407--423, 1990.

\bibitem{gaubenezheng}
S.~Gaubert, W.~McEneaney, and Z.~Qu.
\newblock Curse of dimensionality reduction in max-plus based approximation
  methods: Theoretical estimates and improved pruning algorithms.
\newblock {\em 50th IEEE Conference on Decision and Control and European
  Control Conference (CDC-ECC)}, pages 1054--1061, 2011.

\bibitem{lecphilgirar12}
P.~Girardeau, V.~Leclere, and A.B. Philpott.
\newblock On the convergence of decomposition methods for multistage stochastic
  convex programs.
\newblock {\em Mathematics of Operations Research}, 40:130--145, 2015.

\bibitem{guiguescoap2013}
V.~Guigues.
\newblock {SDDP for some interstage dependent risk-averse problems and
  application to hydro-thermal planning}.
\newblock {\em Computational Optimization and Applications}, 57:167--203, 2014.

\bibitem{guiguessiopt2016}
V.~Guigues.
\newblock Convergence analysis of sampling-based decomposition methods for
  risk-averse multistage stochastic convex programs.
\newblock {\em SIAM Journal on Optimization}, 26:2468--2494, 2016.

\bibitem{guiguesejor2017}
V.~Guigues.
\newblock Dual dynamic programing with cut selection: Convergence proof and
  numerical experiments.
\newblock {\em European Journal of Operational Research}, 258:47--57, 2017.

\bibitem{guiguesrom10}
V.~Guigues and W.~R\"omisch.
\newblock Sampling-based decomposition methods for multistage stochastic
  programs based on extended polyhedral risk measures.
\newblock {\em SIAM J. Optim.}, 22:286--312, 2012.

\bibitem{guiguesrom12}
V.~Guigues and W.~R\"omisch.
\newblock {SDDP for multistage stochastic linear programs based on spectral
  risk measures}.
\newblock {\em Operations Research Letters}, 40:313--318, 2012.

\bibitem{guilejtekregsddp}
V.~Guigues, W.~Tekaya, and M.~Lejeune.
\newblock Regularized decomposition methods for deterministic and stochastic
  convex optimization and application to portfolio selection with direct
  transaction and market impact costs.
\newblock {\em Optimization OnLine}, 2017.

\bibitem{resa}
M.~Hindsberger and A.~B. Philpott.
\newblock Resa: A method for solving multi-stage stochastic linear programs.
\newblock {\em SPIX Stochastic Programming Symposium}, 2001.

\bibitem{morton}
G.~Infanger and D.~Morton.
\newblock Cut sharing for multistage stochastic linear programs with interstage
  dependency.
\newblock {\em Math. Program.}, 75:241--256, 1996.

\bibitem{kozmikmorton}
V.~Kozmik and D.P. Morton.
\newblock {Evaluating policies in risk-averse multi-stage stochastic
  programming}.
\newblock {\em {Mathematical Programming}}, 152:275--300, 2015.

\bibitem{lohndorf}
N.~L\"ohndorf, D.~Wozabal, and S.~Minner.
\newblock Optimizing trading decisions for hydro storage systems using
  approximate dual dynamic programming.
\newblock {\em Operations Research}, 61:810--823, 2013.

\bibitem{mcdesgaube}
W.M. McEneaney, A.~Deshpande, and S.~Gaubert.
\newblock {Curse of complexity attenuation in the curse of dimensionality free
  method for HJB PDEs}.
\newblock {\em American Control Conference}, pages 4684--4690, 2008.

\bibitem{pereira}
M.V.F. Pereira and L.M.V.G Pinto.
\newblock Multi-stage stochastic optimization applied to energy planning.
\newblock {\em Math. Program.}, 52:359--375, 1991.

\bibitem{pfeifferetalcuts}
Laurent Pfeiffer, Romain Apparigliato, and Sophie Auchapt.
\newblock Two methods of pruning benders' cuts and their application to the
  management of a gas portfolio.
\newblock {\em Research Report RR-8133, hal-00753578}, 2012.

\bibitem{philpmatos}
A.~Philpott and V.~de~Matos.
\newblock Dynamic sampling algorithms for multi-stage stochastic programs with
  risk aversion.
\newblock {\em European Journal of Operational Research}, 218:470--483, 2012.

\bibitem{dpcuts0}
A.~Philpott, V.~de~Matos, and E.~Finardi.
\newblock Improving the performance of stochastic dual dynamic programming.
\newblock {\em Journal of Computational and Applied Mathematics}, 290:196--208,
  2015.

\bibitem{philpot}
A.~B. Philpott and Z.~Guan.
\newblock On the convergence of stochastic dual dynamic programming and related
  methods.
\newblock {\em Oper. Res. Lett.}, 36:450--455, 2008.

\bibitem{powellbook}
W.P. Powell.
\newblock {\em {Approximate Dynamic Programming}}.
\newblock John Wiley and Sons, 2nd edition, 2011.

\bibitem{rusparallel93}
A.~Ruszczy\'nski.
\newblock Parallel decomposition of multistage stochastic programming problems.
\newblock {\em Math. Programming}, 58:201--228, 1993.

\bibitem{shapsddp}
A.~Shapiro.
\newblock Analysis of stochastic dual dynamic programming method.
\newblock {\em European Journal of Operational Research}, 209:63--72, 2011.

\bibitem{shadenrbook}
A.~Shapiro, D.~Dentcheva, and A.~Ruszczy\'nski.
\newblock {\em {Lectures on Stochastic Programming: Modeling and Theory}}.
\newblock SIAM, Philadelphia, 2009.

\bibitem{shaptekaya}
A.~Shapiro, W.~Tekaya, J.P. da~Costa, and M.P. Soares.
\newblock {Risk neutral and risk averse stochastic dual dynamic programming
  method}.
\newblock {\em European Journal of Operational Research}, 224:375--391, 2013.

\bibitem{shaptekaya2}
A.~Shapiro, W.~Tekaya, J.P. da~Costa, and M.P. Soares.
\newblock {Worst-case-expectation approach to optimization under uncertainty}.
\newblock {\em Oper. Res.}, 61:1435--1449, 2013.

\bibitem{zhang2016}
W.~Zhang, H.~Rahimian, and G.~Bayraksan.
\newblock Decomposition algorithms for risk-averse multistage stochastic
  programs with application to water allocation under uncertainty.
\newblock {\em INFORMS Journal on Computing}, 28:385--404, 2016.

\end{thebibliography}

\section*{Appendix}

\begin{figure}
\begin{tabular}{|c|c|}
 \hline 
Multicut Level 1 & MLM Level 1 \\
\hline
\begin{tabular}{l}
$I_{t j}^{k}=\{k\}$, $m_{t  j}^k =\mathcal{C}_{t j}^k ( x_{t-1}^k )$.\\
{\textbf{For}} $\ell=1,\ldots,k-1$,\\
\hspace*{0.3cm}{\textbf{If }}$\mathcal{C}_{t j}^k( x_{t-1}^{\ell} ) > m_{t j}^{\ell} + \varepsilon_0 \max(1, |m_{t j}^{\ell}| )$ \\
\hspace*{0.6cm}$I_{t j}^{\ell}=\{k\},\; m_{t j}^{\ell}=\mathcal{C}_{t j}^k( x_{t-1}^{\ell} )$\\
\hspace*{0.3cm}{\textbf{Else if}} $|\mathcal{C}_{t j}^k( x_{t-1}^{\ell} ) - m_{t j}^{\ell}| \leq \varepsilon_0 \max(1, |m_{t j}^{\ell}| )$ \\
\hspace*{0.6cm}$I_{t j}^{\ell}=I_{t j}^{\ell} \cup \{k\}$\\
\hspace*{0.3cm}{\textbf{End If}}\\
\hspace*{0.3cm}{\textbf{If }}$\mathcal{C}_{t j}^{\ell}( x_{t-1}^{k} ) > m_{t j}^k + \varepsilon_0 \max(1, |m_{t j}^{k}| )$ \\
\hspace*{0.6cm}$I_{t j}^k =\{\ell\},\; m_{t j}^k = \mathcal{C}_{t j}^{\ell}( x_{t-1}^{k} )$\\
\hspace*{0.3cm}{\textbf{Else if}} $|\mathcal{C}_{t j}^{\ell}( x_{t-1}^{k} ) - m_{t j}^{k}| \leq \varepsilon_0 \max(1, |m_{t j}^{k}| )$ \\
\hspace*{0.6cm}$I_{t j}^{k}=I_{t j}^{k} \cup \{\ell\}$\\
\hspace*{0.3cm}{\textbf{End If}}\\
{\textbf{End For}}\\
\end{tabular}
&
\begin{tabular}{l}
\vspace*{-0.42cm}\\
$I_{t j}^{k}=\{1\}$, $m_{t  j}^k =\mathcal{C}_{t j}^1 ( x_{t-1}^k )$.\\
{\textbf{For}} $\ell=1,\ldots,k-1$,\\
\hspace*{0.3cm}{\textbf{If }}$\mathcal{C}_{t j}^k( x_{t-1}^{\ell} ) > m_{t j}^{\ell} + \varepsilon_0 \max(1, |m_{t j}^{\ell}| )$ \\
\hspace*{0.6cm}$I_{t j}^{\ell}=\{k\},\; m_{t j}^{\ell}=\mathcal{C}_{t j}^k( x_{t-1}^{\ell} )$\\
\hspace*{0.3cm}{\textbf{End If}}\\
\vspace*{0.4cm}\\
\hspace*{0.3cm}{\textbf{If }}$\mathcal{C}_{t j}^{\ell +1}( x_{t-1}^{k} ) > m_{t j}^k + \varepsilon_0 \max(1, |m_{t j}^{k}| )$ \\
\hspace*{0.6cm}$I_{t j}^k =\{\ell +1\},\; m_{t j}^k = \mathcal{C}_{t j}^{\ell + 1}( x_{t-1}^{k} )$\\
\hspace*{0.3cm}{\textbf{End If}}\\
\vspace*{0.4cm}\\
{\textbf{End For}}\\
\end{tabular}
\\
\hline
\end{tabular}
\caption{Pseudo-codes for selecting the cuts using Multicut Level 1 and MLM Level 1 taking into account approximation errors.}
\label{figurecut2}
\end{figure}

\end{document}